\documentclass[letterpaper, 11pt]{amsart}
\usepackage{amsmath,amssymb,amsxtra,amsthm, amstext, amscd,amsfonts,fancyhdr, hyperref,enumerate, longtable}
\hypersetup{
colorlinks=true,
urlcolor=black,
citecolor=blue,
linkcolor=blue,
}

\usepackage[OT2,T1]{fontenc}
\DeclareSymbolFont{cyrletters}{OT2}{wncyr}{m}{n}
\DeclareMathSymbol{\Sha}{\mathalpha}{cyrletters}{"58}

\newcommand{\ppmod}{\hspace{-2mm}\pmod}

\newcommand{\bC}{{\mathbb{C}}}

\newcommand{\bF}{{\mathbb{F}}}

\newcommand{\bN}{{\mathbb{N}}}

\newcommand{\bQ}{{\mathbb{Q}}}
\newcommand{\bR}{{\mathbb{R}}}

\newcommand{\bZ}{{\mathbb{Z}}}

  \newcommand{\A}{{\mathcal{A}}}
  
  \newcommand{\C}{{\mathcal{C}}}

  \newcommand{\J}{{\mathcal{J}}}
    
\renewcommand{\L}{{\mathcal{L}}}

\renewcommand{\P}{{\mathcal{P}}}
  
  \newcommand{\R}{{\mathcal{R}}}
\renewcommand{\S}{{\mathcal{S}}}

\newcommand{\AND}{\text{ and }}

\newcommand{\OR}{\text{ or }}

\newcommand{\fr}{\mathfrak{r}}
\newcommand{\fs}{\mathfrak{s}}

\newcommand{\Gal}{\operatorname{Gal}}

\newcommand{\GL}{\operatorname{GL}}

\newcommand{\ep}{\varepsilon}

\newcommand{\ol}{\overline}

\newcommand{\upchi}{{\raise.35ex\hbox{$\chi$}}}

\newcommand{\Vol}{\operatorname{Vol}}
\newcommand{\Area}{\operatorname{Area}}

\newcommand{\mmax}{\mathrm{\tiny max}}

\makeatletter
\@namedef{subjclassname@2010}{%
  \textup{2010} Mathematics Subject Classification}
\makeatother

\newtheorem{theorem}{Theorem}[section]
\newtheorem{corollary}[theorem]{Corollary}
\newtheorem{proposition}[theorem]{Proposition}
\newtheorem{lemma}[theorem]{Lemma}

\theoremstyle{definition}
\newtheorem{definition}[theorem]{Definition}

\newtheorem{example}[theorem]{Example}

\theoremstyle{remark}
\newtheorem{remark}[theorem]{Remark}

\numberwithin{equation}{section}

\allowdisplaybreaks
\begin{document}

\title[The number of quartic $D_4$-fields with monogenic cubic resolvent]{The number of quartic $D_4$-fields with monogenic\\ cubic resolvent ordered by conductor}

\author{Cindy (Sin Yi) Tsang}
\address{School of Mathematics (Zhuhai)\\ Sun Yat-Sen University \\
Tangjiawan, Zhuhai\\
Guangdong, 519082, China}
\email{zengshy26@mail.sysu.edu.cn}

\author{Stanley Yao Xiao}
\address{Department of Mathematics\\
University of Toronto \\
Bahen Centre, 40 St. George Steet, Room 6290 \\
Toronto, Ontario, Canada M5S 2E4}
\email{syxiao@math.toronto.edu}
\indent


\begin{abstract} In this paper, we consider maximal and irreducible quartic orders which arise from integral binary quartic forms, via the construction of Birch and Merriman, and whose field of fractions is a quartic $D_4$-field. By a theorem of Wood, such quartic orders may be regarded as quartic $D_4$-fields whose ring of integers has a monogenic cubic resolvent. We shall determine the asymptotic number of such objects when ordered by conductor. We shall also give a lower bound, which we suspect has the correct order of magnitude, and a slightly larger upper bound for the number of such objects when ordered by discriminant. A simplified version of the techniques used allows us to give a count for those elliptic curves with a marked rational 2-torsion point when ordered by discriminant.
\end{abstract}

\maketitle

\tableofcontents


\section{Introduction}
\label{Intro}

In his groundbreaking work \cite{HCL3}, Bhargava showed that isomorphism classes of pairs
\[\label{QR}(Q,R),\mbox{ where $Q$ is a quartic ring and $R$ is a cubic resolvent of $Q$},\]
may be parametrized by $\GL_2(\bZ)\times\GL_3(\bZ)$-orbits of pairs $(U,V)$ of integral ternary quadratic forms, where the action is defined as in \cite[(11)]{HCL3}. He also showed that maximal quartic orders have a unique cubic resolvent; see \cite[Corollary 5]{HCL3}. Hence, when restricted to irreducible and maximal quartic orders, this may be regarded as a parametrization of quartic fields. \\

It is known, by work of Birch-Merriman \cite{BM} and Nakagawa \cite{Nakagawa}, that integral binary forms of degree $d$ may be used to construct $d$-ic orders. More specifically, consider
\[ F(x,y) = a_dx^d + a_{d-1}x^{d-1}y + \cdots + a_1xy^{d-1} + a_0y^d,\mbox{ where }a_i\in\bZ,\]
and assume for simplicity that $F$ is irreducible. Let $\theta_F$ be a root of $F(x,1)$. Then, the $d$-ic order $Q_F$ associated to $F$ is the subring of $L_F = \bQ(\theta_F)$ with $\bZ$-module basis given by
\[\zeta_0 = 1,\,\ \zeta_1 = a_d\theta_F,\,\ \zeta_2 = a_d\theta_F^2+a_{d-1}\theta_F, \,\ \cdots\,\ \zeta_{d-1} = a_d\theta_F^{d-1}+\cdots+a_{2}\theta_F.\]
By \cite{Nakagawa}, it is known that $Q_F$ is indeed a unital ring, and we have $\Delta(F) = \mbox{disc}(Q_F)$. Further, the isomorphism class of $Q_F$ corresponds to the $\GL_2(\bZ)$-equivalence class of $F$. For $d=3$, by work of Delone-Faddeev \cite{DF}, all cubic orders arise in this way. For $d \geq 4$, it is expected that most $d$-ic orders do not arise this way. Nevertheless, for a long time, this construction was essentially the only one to obtain irreducible orders of high rank over $\bZ$. \\

In \cite{Wood}, Wood further showed that integral binary quartic forms 
\begin{equation}\label{generic F} F(x,y) = a_4x^4+a_3x^3y+a_2x^2y^2+a_1xy^3+a_0y^4,\mbox{ where }a_i\in\bZ,\end{equation}
correspond precisely to the pairs
\[(Q,R),\mbox{ where $Q$ is a quartic ring and $R$ is a monogenic cubic resolvent of $Q$}.\]
Let us write ternary quadratic forms as matrices, via
\[ u_{11}x^2 + u_{22}y^2 + u_{33}z^2 + u_{12}xy + u_{13}xz + u_{23}yz \longleftrightarrow \begin{pmatrix} u_{11} & \frac{u_{12}}{2} & \frac{u_{13}}{2} \\ \frac{u_{12}}{2} & u_{22} & \frac{u_{23}}{2} \\ \frac{u_{13}}{2} & \frac{u_{23}}{2} & u_{33}\end{pmatrix}.\]
Then, the above correspondence may be made explicit by sending $F$ to the pair $(Q_F,R_F)$ which is associated to the ternary quadratic forms
\[ \left(U_0,V_F\right) = \left(\begin{pmatrix} 0 & \frac{-1}{2} & 0 \\ \frac{-1}{2} & 0 & 0\\  0 & 0 & 1
\end{pmatrix},\begin{pmatrix} a_0 & 0 & \frac{a_1}{2} \\ 0 & a_4 & \frac{a_3}{2} \\ \frac{a_1}{2} & \frac{a_3}{2} & a_2
\end{pmatrix}\right).\]
Now, the ring of integers of any quartic field is a maximal quartic order and hence has a unique cubic resolvent. In view of these observations, we are prompted to make the following definition:

\begin{definition} \label{qmono} A quartic field is said to \emph{have monogenic cubic resolvent} if the cubic resolvent of its ring of integers is monogenic.
\end{definition}

In particular, we have a natural bijection
\begin{equation}\label{correspondence}\frac{\left\{\begin{array}{c}\mbox{integral and irreducible binary}\\\mbox{quartic forms $F$ with $Q_F$ maximal}\end{array} \right\}}{\GL_2(\bZ)\mbox{-equivalence}}\longleftrightarrow
\frac{\left\{\begin{array}{c}\mbox{quartic fields $L$ with}\\\mbox{monogenic cubic resolvent}\end{array}\right\}}{\mbox{field isomorphism}}.\end{equation}
We shall write $\Gal(F)$ for the Galois group of the Galois closure of $L_F = \bQ(\theta_F)$ over $\bQ$, and for any group $G$, we say that $L$ is a \emph{$G$-field} if its Galois closure over $\bQ$ has Galois group isomorphic to $G$. The main purpose of this paper is to enumerate the objects above when $\Gal(F)\simeq D_4$, or equivalently, when $L$ is a quartic $D_4$-field. Here $D_4$ denotes the dihedral group of order eight.\\

Recall that by the Delone-Faddeev correspondence \cite{DF}, isomorphism classes of cubic orders are in bijection with $\GL_2(\bZ)$-equivalence classes of integral binary cubic forms. A cubic order is monogenic if and only if the corresponding $\GL_2(\bZ)$-equivalence class contains a \emph{monic} cubic form, that is, a cubic form which represents $1$ integrally. For $d \geq 2$, we expect that most $\GL_2(\bZ)$-equivalence classes of integral binary forms of degree $d$, when ordered by discriminant, do not represent $1$ integrally. Thus, one expects that the objects considered in Definition \ref{qmono} are rare among all quartic fields.

\begin{example}Consider the pair $(Q,R)$, where $Q$ is a quartic ring and $R$ is a cubic resolvent of $Q$, corresponding to the pair 
\[(U,V) = \left(\begin{pmatrix} 0 & 0 & 0 \\ 0 & 1 & \frac{-1}{2} \\ 0 & \frac{-1}{2} & -1 \end{pmatrix}, \begin{pmatrix} 1 & 0 & \frac{1}{2} \\ 0 & -1 & \frac{-1}{2} \\ \frac{1}{2} & \frac{-1}{2} & 4 \end{pmatrix} \right).\]
The associated cubic resolvent is given by 
\[ R(x,y) = 4\det(Ux + Vy) = -y(5x^2 - 17xy + 16y^2),\]
which does not represent $1$ integrally because there is no $x\in\bZ$ for which $5x^2 - 17x + 16 = \pm 1$. This implies that $R$ is not monogenic. Using the criteria given by Bhargava \cite{HCL3} (see Proposition \ref{maximality types} below), it may be verified that $Q$ is a maximal quartic ring; note that $\mathrm{disc}(Q)$, which equals $\mathrm{disc}(R(x,y))$, is $-31\cdot 5^2$ and so it suffices to check maximality at the primes $5$ and $31$. \\

Let us remark (mostly due to the lack of a reference) that, when $Q$ is irreducible, the field of definition of the intersection points of the two ternary quadratic forms $U(x,y,z)$ and $V(x,y,z)$ is exactly the Galois closure of the field of fractions of $Q$. This can be seen easily by first taking a subring $Q^\prime$ with monogenic cubic resolvent of $Q$ so that $Q^\prime$ corresponds to a binary quartic form. \\

We now compute the intersection points over $\bC$ of the conics defined by
\[ U(x,y,z) =  y^2 - yz - z^2\AND V(x,y,z) = x^2 + xz - y^2 - yz + 4z^2.\]
The set of points on the conic defined by $U(x,y,z)$ is given by
\[\left\{\left(x, \frac{1 + \sqrt{5}}{2} z, z \right) : x, z \in \bC \right\} \cup \left\{\left(x, \frac{1 - \sqrt{5}}{2} z, z \right) : x,z \in \bC \right\}. \]
Substituting the relation $y = (1 \pm \sqrt{5}) z/2$ into $V(x,y,z)$ gives
\[ x^2 + xz - \left(\frac{1 \pm \sqrt{5}}{2}\right)^2 z^2 - \left(\frac{1 \pm \sqrt{5}}{2}\right) z^2 + 4z^2 = x^2 + xz + (2 \mp \sqrt{5}) z^2.\]
Solving for all cases we get the representatives 
\[\left(\frac{-1 + \sqrt{-7 + 4\sqrt{5}}}{2}  , \frac{1 + \sqrt{5}}{2} ,1\right), \left(\frac{-1 - \sqrt{-7 + 4\sqrt{5}}}{2}  , \frac{1 + \sqrt{5}}{2} ,1\right), \]
\[\left(\frac{-1 + \sqrt{-7 - 4\sqrt{5}}}{2}  , \frac{1 - \sqrt{5}}{2} ,1\right), \left(\frac{-1 - \sqrt{-7 - 4\sqrt{5}}}{2}  , \frac{1 - \sqrt{5}}{2} ,1\right). \]
Since $-7 - 4 \sqrt{5} < 0$, it cannot be a square in $\bQ(\sqrt{5})$, hence the extension $\bQ\left(\sqrt{-7 - 4 \sqrt{5}}\right)$ has degree $4$ over $\bQ$. It is standard to check that its Galois closure has Galois group $D_4$. This implies that the pair $(U,V)$ corresponds to a quartic $D_4$-field whose cubic resolvent ring is not monogenic.
\end{example}

In \cite{ASVW}, Altug-Shankar-Varma-Wilson counted all quartic $D_4$-fields $L$ by their \emph{conductor}, which as in \cite[Proposition 2.4]{ASVW} may be defined as
\[ \label{ASVW cond} \C(L) = \mbox{disc}(L)/\mbox{disc}(K_L),\]
where $K_L$ denotes the unique quadratic subfield of $L$. They showed that the number of isomorphism classes of quartic $D_4$-fields $L$ with $|\C(L)| < X$ is asymptotic to $\kappa \cdot X \log X$, where $\kappa$ is an explicit constant given by an Euler product. We note that the above definition of conductor applies equally well to quartic orders with a unique quadratic subalgebra, and we shall give the corresponding definition $\C(F)$ for integral binary quartic forms $F$ in Subsection~\ref{conductor poly}.\\

Motivated by the work of \cite{ASVW}, let $N_{D_4}^{(r_2)}(X)$ be the number of isomorphism classes of quartic $D_4$-fields $L$ with monogenic cubic resolvent such that $L$ has exactly $4-2r_2$ real embeddings and $|\C(L)|<X$. By (\ref{correspondence}), equivalently $N_{D_4}^{(r_2)}(X)$ is the number of $\GL_2(\bZ)$-equivalence classes of integral and irreducible binary quartic forms $F$ with $Q_F$ maximal and $\Gal(F)\simeq D_4$ such that $F(x,1)$ has exactly $4-2r_2$ real roots and $|\C(F)|<X$. Our main theorem is the following:

\begin{theorem} \label{real split MT}We have
\[N_{D_4}^{(r_2)}(X)  =\frac{\fr(r_2)}{\zeta(2)}\frac{\sqrt{2}\Gamma(1/4)^2}{48 \sqrt{\pi}} \prod_p \left(1 - \frac{2p-1}{p^3}\right) X^{3/4} \log X + O \left(X^{3/4} (\log X)^{3/4} \right),\]
where $\fr(0) = 1$, $\fr(1)=\sqrt{2}$, and $\fr(2) = 1+\sqrt{2}$.
\end{theorem}

\begin{remark}The quality of the error term in Theorem~\ref{real split MT} is merely a log-power saving, and this is mostly an artifact of our proof.  \end{remark}

\begin{remark}\label{product P}The real number in Theorem \ref{real split MT} given by the absolutely convergent product
\[\P=\prod_{p}\left(1 - \frac{2p-1}{p^3}\right)\]
is the so-called \emph{carefree constant}. It arises from counting \emph{carefree couples}, namely pairs $(a,b)$ of coprime positive integers with $a$ square-free. In particular, the number of carefree couples $(a,b)$ with $a,b\leq X$ is equal to $\P X^2 + O(X\log X)$.\end{remark}

\begin{remark} Although Theorem \ref{real split MT} is given in terms of real densities, for any finite set of primes $\mathfrak{P}$ we can impose any number of congruence conditions modulo $p^k$ for $p \in \mathfrak{P}$ and the asymptotic formula in the theorem will still be valid (even if the main term becomes zero). This is because our square-free sieve provides the necessary tail estimate. In particular, we can count the number of such quartic fields having specified splitting behaviour at any finite number of primes. 
\end{remark}

Comparing Theorem \ref{real split MT} with the main theorem of \cite{ASVW} we see two features. Firstly, Theorem \ref{real split MT} provides an asymptotic formula for a subfamily which is \emph{zero-density}. In general, one expects that the distribution behavior in such thin families to be different from a `large' family (having positive density). Secondly, our Theorem \ref{real split MT} preserves an essential feature of the main theorem in \cite{ASVW}: the main term can be given in terms of a natural Euler product. \\

It is interesting to note that when sorted by the number of real embeddings, the distribution of all quartic $D_4$-fields and that of quartic $D_4$-fields with monogenic cubic resolvent do not agree. As can be seen from \cite{ASVW}, the proportions of all quartic $D_4$-fields having $4,2,$ and $0$ real embeddings ordered by conductor are $1/3$, $1/2$, and $1/6$, respectively. However, by Theorem~\ref{real split MT} above, the proportions of quartic $D_4$-fields with monogenic cubic resolvent having $4,2,$ and $0$ real embeddings are $20.710 \dots \%,29.289 \dots \%$, and $1/2$, respectively. Thus, Theorem \ref{real split MT} can be viewed as a demonstration of the phenomenon that \emph{global restrictions} on sets of rings and fields change the Cohen-Lenstra heuristics in ways that local restrictions do not. Indeed, we expect that the distribution of monogenic rings to be different than the distribution of generic rings of fixed rank. \\

It is natural to ask what happens when the quartic fields with monogenic cubic resolvent are ordered by discriminant rather than conductor. If we count quartic $V_4$-fields, where $V_4$ denotes the Klein group, then we can prove an exact asymptotic formula using a result of Stewart and the second author \cite{SX}. However, if we count quartic $D_4$-fields, then we are only able to get a lower bound, which we suspect has the correct order of magnitude, and a slightly larger upper bound. For $G\in\{V_4,D_4\}$, let $N_G'(X)$ be the number of isomorphism classes of quartic $G$-fields $L$ with monogenic cubic resolvent such that $|\mbox{disc}(L)|<X$. Again by (\ref{correspondence}), equivalently $N_G'(X)$ is the number of $\GL_2(\bZ)$-equivalence classes of integral and irreducible binary quartic forms $F$ with $Q_F$ maximal and $\Gal(F)\simeq G$ such that $|\Delta(F)|<X$. Then, we have the following theorems:

\begin{theorem} \label{V4 thm} We have
\[N_{V_4}'(X) = \frac{7}{8}\frac{\Gamma(1/3)^2}{\Gamma(2/3)} \prod_p \left(1 - \frac{4p-3}{p^3}\right) X^{1/3} + O \left(X^{1/3} (\log X \cdot \log \log X)^{-1} \right).\]
\end{theorem}

\begin{theorem} \label{D4 thm} We have
\[X^{1/2} (\log X)^2 \ll N_{D_4}'(X) \ll X^{1/2 + 1/\log\log X} \log X.\]
\end{theorem}

Note that Theorem \ref{V4 thm} can be compared with a theorem of Baily in \cite{Bai}. He obtained the asymptotic formula for the number of quartic $V_4$-fields with discriminant bounded by $X$, obtaining a main term of the shape $\kappa^\prime \cdot X^{1/2} (\log X)^2$ for an explicit constant $\kappa^\prime$ given by an Euler product. We again see that Theorem \ref{V4 thm} gives an asymptotic formula for a thin family. \\ 

Similarly, one should compare Theorem \ref{D4 thm} with the main result of Cohen, Diaz y Diaz, and Olivier \cite{CDO}. They established an asymptotic formula for the number of quartic $D_4$-fields ordered by discriminant using $L$-function methods, and the main term is of the shape $\kappa^{\prime\prime}\cdot X$. What one should note from their work is that it is unlikely that the constant $\kappa^{\prime\prime}$ can be given in terms of an Euler product (i.e., a product of local densities), and perhaps one should expect the same for quartic $D_4$-fields with monogenic cubic resolvent. We note again that Theorem \ref{D4 thm} provides an asymptotic estimation of elements in a thin family. \\ 


Next let us compare the orders of magnitude of the asymptotic number of quartic $D_4$-fields (with monogenic cubic resolvent), when ordered by discriminant and by conductor, respectively. When counting all quartic $D_4$-fields, by the results of \cite{ASVW} and \cite{CDO} mentioned above, the ratio of the two main terms is of order $\log X$. When counting only those with monogenic cubic resolvent, by our Theorems \ref{real split MT} and \ref{D4 thm}, the ratio between the main terms is a bit more dramatic, as they are not the same size on the logarithmic scale. This can be explained heuristically by the following: the discriminant of a pair of generic ternary quadratic forms has degree equal to the number of parameters (both equal to $12$), which gives an exponent of $1$ on the log scale. The conductor, as given in \cite{ASVW}, is a degree $8$ polynomial in $8$ parameters, which again gives the exponent $1$. In our case the discriminant is a degree $6$ polynomial in $3$ variables, giving the exponent $1/2$ while the conductor is a degree $4$ polynomial in $3$ variables, giving the exponent $3/4$. \\

Finally, as was shown by Bhargava and Shankar in \cite{BS}, there is a deep connection between binary quartic forms and elliptic curves. In particular, integral and irreducible binary quartic forms $F$ with \emph{small} Galois group, by which we mean $\Gal(F)\simeq D_4,V_4,C_4$, are related to elliptic curves over $\bQ$ with rational $2$-torsion; see the last paragraph in \cite[Section 1]{TX}.
Here $C_4$ is the cyclic group of order $4$. Such elliptic curves are given by the models 
\[ \label{2tor curve} E_{a,b}: y^2 = x(x^2 + ax + b), \mbox{ where }a,b\in\bZ\mbox{ and }p^2\mid a\mbox{ implies }p^4\nmid b\mbox{ for all primes $p$},\]
which can be viewed as a set of elliptic curves with a \emph{marked rational 2-torsion point}.  \\

Let $N_{E}'(X)$ be the number of such elliptic curves $E_{a,b}$ with $0<|\mbox{disc}(E_{a,b})|<X$, and note that $\mbox{disc}(E_{a,b}) = b^2(a^2-4b)$. Using a simplified version of the proof of Theorem~\ref{real split MT}, we are able to obtain an exact asymptotic formula:

\begin{theorem} \label{EC thm} We have
\[N_E'(X) = \frac{2}{3}\prod_{p}\left(1-\frac{1}{p^6}\right) X^{1/2} \log X+ O \left(X^{1/2}\right)=\frac{2}{3}\frac{1}{\zeta(6)}X^{1/2} \log X+ O \left(X^{1/2}\right).\]
\end{theorem}

\begin{remark}As we shall explain in Subsection \ref{structure}, to count quartic $D_4$-fields with monogenic cubic resolvent, we need to count certain triplets of integers, and we do so by counting lattice points in a $2$-dimensional region, weighing each point by an appropriate 
divisor function. The major obstacle which prevents us from obtaining an exact asymptotic formula for $N_{D_4}'(X)$ is that the divisor function $d(\cdot)$ is very poorly understood in short intervals. We are unable to average and are forced to use the uniform upper bound $d(n) = O(n^{1/\log\log n})$. However, for $N_E'(X)$ the region is already $2$-dimensional, and so we do not need to take divisors. The obstacle is therefore circumvented and we are able to get an exact asymptotic formula.
\end{remark}

One should compare Theorem \ref{EC thm} with the analogous result of counting elliptic curves with a marked $2$-torsion point by \emph{height}, used by Klagsbrun and Lemke-Oliver in \cite{KL-O} to compute the distribution of Tamagawa ratios in this family. In particular, they showed that the number of such elliptic curves having height bounded by $X$ is of order $X^{1/2}$. They used this to obtain the interesting consequence that the average size of the $\phi$-Selmer group, given by the degree $2$-isogeny $\phi$ whose square is multiplication-by-$2$, is unbounded in this family, in marked contrast to the generic case considered by Bhargava and Shankar \cite{BS}. Theorem \ref{EC thm} shows that there is an extra factor of $\log X$ when counting by discriminant. \\ 

Perhaps one hopes that the analogous result to Theorem \ref{real split MT} can be obtained for elliptic curves; that is, one can count the curves $E_{a,b}$ by their conductor. However, unlike the case of binary quartic forms, we do not have an easy way to associate coefficients $a,b$ in the model $E_{a,b}$ with the conductor of the elliptic curve. It can easily be shown that when $b(a^2 - 4b)$ is square-free that it is indeed equal to the conductor of $E_{a,b}$; but in general this is not true. Thus, our methods do not yield a simple way to count elliptic curves with rational 2-torsion by conductor. \\

Finally, we remark that the proportionality constants in Theorems \ref{real split MT}, \ref{V4 thm}, and \ref{EC thm} are exactly as one would expect: they are products of local densities. The infinite Euler product that appears in each case captures the local behavior at every prime; for the ring cases it captures the behavior of being maximal at a prime, and for the elliptic curve case it captures the property of being minimal. The various terms involving the $\Gamma$-function and other irrational constants reflect the real density of a certain region in $\bR^2$. This is in marked contrast to the main term in \cite{CDO}, where one does not expect it to have a decomposition as a product of local densities. We also note that the local densities in Theorems \ref{real split MT} and \ref{V4 thm} are different than those that appear in \cite{ASVW} and \cite{Bai}, respectively, for the full families of quartic $D_4$- and $V_4$-fields. 

\subsection{Structure of the argument}
\label{structure}

This paper can be thought of as separated into three components. First, we show that the arithmetic objects of interest, namely $\GL_2(\bZ)$-equivalence classes of integral and irreducible binary quartic forms $F$ with the property that $L_F=\bQ(\theta_F)$ is a quartic $D_4$-field and $Q_F$ is the ring of integers in $L_F$, are parametrized by some subsets of the families 
\begin{align}\label{F in J}
V_{\J^{(1)}}(\bZ) & = \{Ax^4 + Bx^2 y^2 + Cy^4 : A,B,C\in\bZ\},\\\notag
V_{\J^{(2)}}(\bZ) & = \{Ax^4 + Bx^3 y + (C+2A)x^2 y^2 + Bxy^3 + A y^4 : A,B,C\in\bZ\},\\\notag
V_{\J^{(3)}}(\bZ) & = \{Ax^4 + Bx^3 y + (C-2A)x^2 y^2 - Bxy^3 + A y^4 : A,B,C\in\bZ\}.
\end{align}
The key idea behind the proof is the following observation: for any such $F$, the field $L_F$ contains a quadratic subfield $K_F$, and $Q_F \cap K_F$ must be equal to the ring of integers in $K_F$. In Theorem \ref{necessary thm'}, we shall use this to give a strong necessary condition for $Q_F$ to be maximal, which can be reduced to the condition that $F$ has a $\GL_2(\bZ)$-translate lying in one of the families in (\ref{F in J}). The proof involves an explicit determination of a basis of the ring $Q_F \cap K_F$ (see (\ref{QF cap K}) below) and comparison of discriminants. Another important ingredient is a characterization for $\Gal(F)$ to be small in terms of binary quadratic forms (see Theorem \ref{h(f,g)}) which builds upon our previous work \cite{TX}. The characterization also allows us to define, for $\Gal(F)\simeq D_4$, the \emph{conductor polynomial} $\C(F)$ which coincides with $\C(L_F)$ at least when $Q_F$ is maximal; see Subsection \ref{conductor poly}, and we shall see that $\C(F)$ is given by an explicit polynomial expression of the coefficients of $F$. When $F$ is in one of the families (\ref{F in J}), the expression of $\C(F)$ (stated in (\ref{Ci})) becomes extremely simple, and this enables us to formulate the counting theorems.\\

The second component involves counting the elements in $V_{\J^{(i)}}(\bZ)$ having bounded conductor. We remark that the forms in $V_{\J^{(3)}}(\bZ)$ give a negligible contribution (see Proposition \ref{N3X}), so it suffices to consider the families $V_{\J^{(i)}}(\bZ)$ for $i=1,2$. The main idea of the proof of the counting theorems is \emph{geometry of numbers}. If we identify each of the families $V_{\J^{(i)}}(\bZ)$ as a subset of $\bR^3$, then the problem is reduced to counting lattice points in some region in $\bR^3$. These regions are relatively skew, induced by the product structure of the height function $\C(F)$. If one tries to apply Davenport's Lemma (see Proposition \ref{Davenport}), then one would need to show that the volumes of these regions actually exceed the areas and lengths of the projections onto the coordinate planes and axes. This issue is complicated by the skewness. \\

Fortunately, the product structure of the counting function $\C(F)$ allows us to reduce the counting problem to one about counting lattice points in a \emph{2-dimensional} region with \emph{weights} (see Subsection \ref{planar section}). In this planar region, it is then a relatively straightforward matter to show that the contribution from the problematic subregions (cusps) are bounded (see Lemma \ref{tail cut}). \\

We then proceed to count the forms $F$ in $V_{\J^{(i)}}(\bZ)$ which are \emph{maximal}, namely it associated quartic ring $Q_F$ is maximal. This necessitates the formulation of a set of criteria to determine whether $F$ is maximal, and we shall distill these conditions from \cite{HCL3} in Section \ref{max section}. These criteria, stated as congruence conditions on the coefficients of $F$, are explicit enough to allow us to calculate the exact density of elements in $V_{\J^{(i)}}(\bZ)$ which are maximal at a given prime $p$. \\

Since being a maximal ring is equivalent to being a maximal ring at every prime, we need to stitch the local conditions together. This is tantamount to applying a square-free sieve to the conductor polynomial (see Subsection \ref{sieve}). As is well-known, the most difficult aspect of applying such an argument is uniformity; that is, one has to show that the contribution from extremely large primes are controlled. We again rely on the fact that the conductor polynomial is reducible and each irreducible component has low degree to accomplish this. However, the skewness of the regions considered makes this task considerably more challenging. \\ 

To complete the proof of Theorem \ref{real split MT} we have to show that most of the forms we count are dihedral, as opposed to being reducible or having a strictly smaller Galois group. This is accomplished in Subsection \ref{max abel forms}. \\

In the final section of the paper we prove our counting results when ordered by discriminant. For Theorem \ref{V4 thm}, we are able to extract the exact asymptotic formula by applying generalization of a theorem of Mahler due to Stewart and the second author \cite{SX}. For Theorem \ref{D4 thm}, we are unable to obtain an asymptotic formula because the region in $\bR^3$ carved out by bounding the discriminant is just too skew, and that the ``cusps" actually contain most of the points! Hence, it is unlikely that this problem can be solved using methods from geometry of numbers without significant new ideas. Nevertheless, we are able to obtain a lower bound which we expect to be the correct order of magnitude and an upper bound that is the same size on the logarithmic scale as the lower bound. Finally, we count elliptic curves with a marked $2$-torsion point and obtain Theorem \ref{EC thm}. The difference in Theorems \ref{D4 thm} and \ref{EC thm} is that the latter involves counting integral points in a \emph{2-dimensional} region without weights, and the skewness there is not as damaging. 

\section{Pairs of binary quadratic forms and pre-maximal binary quartic forms}\label{reduce section}

In \cite{TX}, we studied integral and irreducible binary quartic forms whose Galois group is small. More precisely, consider the \emph{twisted action} of $\GL_2(\bR)$ on real binary quartic forms, given by
\begin{equation}\label{action} F_T(x,y) = \frac{1}{\det(T)^2}F(t_1x+t_2y,t_3x+t_4y)\mbox{ for }T = \begin{pmatrix}t_1&t_2\\t_3&t_4\end{pmatrix}.\end{equation}
For each real binary quadratic form $\J$ having non-zero discriminant, put
\begin{align*}\label{V def}
V_\J(\bR) & = \{\mbox{real binary quartic forms $F$ such that $F_{M_\J} = F$}\},\\\notag
V_\J(\bZ) & = \{\mbox{integral binary quartic forms $F$ such that $F_{M_\J} = F$}\},
\end{align*}
where $M_\J$ is the matrix defined by
\begin{equation}\label{MJ def}M_\J = \begin{pmatrix} \beta & 2\gamma \\ -2\alpha & -\beta\end{pmatrix}\mbox{ for }\J(x,y) = \alpha x^2+\beta xy + \gamma y^2.\end{equation}
The definitions of $V_\J(\bR)$ and $V_{\J}(\bZ)$ clearly remain unchanged if we scale $\J$ by an element of $\bR^\times$. In \cite[Theorem 1.1]{TX}, we proved: 

\begin{proposition}\label{criterion}Let $F$ be an integral and irreducible binary quartic form. Then, its Galois group is small if and only if $F\in V_\J(\bZ)$ for some integral binary quadratic form $\J$ having non-zero discriminant. Moreover, in this case, up to scaling there is a unique such $\J$ if  $\Gal(F)\simeq D_4,C_4$, and exactly three such $\J$ if $\Gal(F)\simeq V_4$.
\end{proposition}

In this section, let us fix an integral and primitive binary quadratic form
\[ \J(x,y) = \alpha x^2 + \beta xy + \gamma y^2, \mbox{ where }\beta^2-4\alpha\gamma\neq0\mbox{ and }\gcd(\alpha,\beta,\gamma)=1,\]
having non-zero discriminant. We shall give an alternative description of $V_\J(\bZ)$ in terms of binary quadratic forms. To that end, we need to make a definition:

\begin{definition}\label{con}Given a real binary quadratic form $\varphi$, define $\varphi_2,\varphi_1,\varphi_0\in\bR$ by
\[\label{varphi210}\varphi(x,y) = \varphi_2 x^2 + \varphi_1 xy + \varphi_0y^2,\mbox{ and write }\Delta(\varphi) = \varphi_1^2 - 4\varphi_2\varphi_0\]
for its discriminant. Given a pair $(f,g)$ of real binary quadratic forms, define
\[\J_{(f,g)}(x,y) = \frac{1
}{2} \begin{vmatrix} \frac{\partial f}{\partial x} & \frac{\partial f}{\partial y} \\ \frac{\partial g}{\partial x} & \frac{\partial g}{\partial y} \end{vmatrix} = (f_2 g_1 - f_1 g_2)x^2 + 2(f_2 g_0 - f_0 g_2)xy + (f_1 g_0 - f_0 g_1)y^2,\]
called the \emph{Jacobian determinant of $(f,g)$}. Also, let the group $\GL_2(\bR)$ act on it via
\[ _{T}(f,g) = \frac{1}{\det(T)}(t_1f + t_2g, t_3f + t_4g)\mbox{ for }T = \begin{pmatrix}t_1&t_2\\t_3&t_4\end{pmatrix}.\]
Note that the Jacobian determinant is invariant under this action. In the case that both $f$ and $g$ are integral, we shall say that $(f,g)$ is \emph{primitive} if 
\[ \Delta(\J_{(f,g)})\neq0\AND\gcd\left(\J_{(f,g),2}, \J_{(f,g),1}/2, \J_{(f,g),0}\right) = 1.\]
These definitions are motivated by Propositions~\ref{Jfg lemma} and~\ref{primitive char} below.
\end{definition}

\begin{theorem}\label{h(f,g)}Let $F$ be an integral binary quartic form of non-zero discriminant. Then, we have $F\in V_\J(\bZ)$ if and only if $F = h(f,g)$ for some integral binary quadratic forms $h,f,g$ such that $\J_{(f,g)}(x,y)$ is proportional to $\J(x,y)$. In this case, we have 
\[\label{thm1 display} (\Delta(h)\Delta(\J_{(f,g)})/4)^2\mbox{ divides }\Delta(F)\]
and the pair $(f,g)$ may be taken to be primitive. Moreover, if $F=h(f,g)$ and $F=h'(f',g')$ are two such representations of $F$ with $(f,g)$ and $(f',g')$ primitive, then $h$ and $h'$ are $\GL_2(\bZ)$-equivalent.
\end{theorem}

Theorem \ref{h(f,g)} can be given in a more pithy manner: an integral and irreducible binary quartic form has small Galois group if and only if it is a composition of integral binary quadratic forms. Theorem~\ref{h(f,g)} also allows us to define the \emph{conductor polynomial} $\C(F)$ of $F$ when $\Gal(F)\simeq D_4$ which gives the conductor of the corresponding quartic $D_4$-field when $Q_F$ is maximal; see Subsection~\ref{conductor poly}. Another important application of Theorem~\ref{h(f,g)} is the next result, which gives a strong necessary condition for $Q_F$ to be maximal when $F$ has small Galois group.

\begin{theorem}\label{necessary thm'}Let $F\in V_{\J}(\bZ)$ be an irreducible form such that $Q_F$ is maximal, and write $F = h(f,g)$ as given by Theorem~\ref{h(f,g)}. Then, necessarily \[\mbox{$\Delta(h)$ is a fundamental discriminant, $(f,g)$ is primitive, and $\Delta(\J_{(f,g)}) = \pm4$.}\] In particular, we must have $\Delta(\J)\in\{1,4,-4\}$.
\end{theorem}

For $D\in\{1,4,-4\}$, up to $\GL_2(\bZ)$-equivalence there is a unique integral and primitive binary quadratic form $\J$ of discriminant $D$. In particular, we may take them to be
\[ \J^{(1)}(x,y) = xy,\,\ \J^{(2)}(x,y) = x^2 - y^2,\,\ \J^{(3)}(x,y) = x^2 + y^2,\]
respectively. A simple calculation shows that $V_{\J^{(i)}}(\bZ)$ is indeed given as in (\ref{F in J}) for all $i=1,2,3$. Clearly $F\mapsto F_T$ defines a bijection $V_{\J}(\bZ)\longrightarrow V_{\J_T}(\bZ)$ preserving the $\GL_2(\bZ)$-action. Theorem \ref{necessary thm'}, together with Proposition~\ref{criterion}, then yields:

\begin{corollary} \label{three fams} Let $F$ be an integral and irreducible binary quartic form with small Galois group such that $Q_F$ is maximal. Then $F$ is $\GL_2(\bZ)$-equivalent to an element of $V_{\J^{(i)}}(\bZ)$ for some $i =1,2,3$, and this $i$ is unique when $\Gal(F)\simeq D_4,C_4$.
\end{corollary}

We remark that Wood obtained a similar result to our Theorems \ref{h(f,g)} and \ref{necessary thm'} in her thesis \cite{Wood thesis}. Indeed, she showed that if a pair $(U,V)$ of ternary quadratic forms corresponds to a maximal quartic $D_4$-ring $Q$, then they must admit a $\GL_2(\bZ) \times \GL_3(\bZ)$-translate of the shape
\[\left(U^\flat, V^\flat \right) = \left(\begin{pmatrix} 0 & 0 & 0 \\ 0 & u_{22} & \frac{u_{23}}{2} \\ 0 & \frac{u_{23}}{2} & u_{33} \end{pmatrix}, \begin{pmatrix} \pm 1 & \frac{v_{12}}{2} & \frac{v_{13}}{2} \\ \frac{v_{12}}{2} & v_{22} & \frac{v_{23}}{2} \\ \frac{v_{13}}{2} & \frac{v_{23}}{2} & v_{33} \end{pmatrix} \right). \]
The matrix $U^\flat$ corresponds to a binary quadratic form $u^\flat(x,y) = u_{22}x^2 + u_{23}xy + u_{33}y^2$, which then gives a quadratic suborder of $Q$, and the condition that the top left entry of $V^\flat$ must be a unit corresponds to a necessary condition for $Q$ to be maximal. In Section \ref{max section}, we shall use criteria given by Bhargava in \cite{HCL3} to obtain maximality criteria for forms in $V_{\J^{(i)}}(\bZ)$, $i = 1,2,3$. These results are also needed by Altug, Shankar, Varma, and Wilson for their work \cite{ASVW}. \\

Finally, in view of Corollary \ref{three fams}, we make the following definition:
\begin{definition} \label{premax def} An integral binary quartic form $F$ is \emph{maximal} if it has maximal quartic ring $Q_F$, and \emph{pre-maximal} if it is $\GL_2(\bZ)$-equivalent to a form in $V_{\J^{(i)}}(\bZ)$ for some $i = 1,2,3$. \end{definition}

Our goal is then to enumerate, for each $i=1,2,3$, the set
\[ V_{\J^{(i)}}^{\mmax}(\bZ) = \{F\in V_{\J^{(i)}}(\bZ): F\mbox{ is maximal}\}.\]
In particular, by Corollary \ref{three fams}, we have
\begin{align*} N_{D_4}^{(r_2)}(X) &= \sum_{i=1}^{3}\#\left\{[F] :\begin{array}{c} \mbox{irreducible $F\in V_{\J^{(i)}}^{\mmax}(\bZ)$ with $\Gal(F)\simeq D_4$ such}\\\mbox{that $F(x,1)$ has exactly $4-2r_2$ real roots and $|\C(F)|<X$}\end{array}\right\},\\[0.5pt]
N_{D_4}'(X) &= \sum_{i=1}^{3}\#\left\{[F] : \mbox{irreducible $F\in V_{\J^{(i)}}^{\mmax}(\bZ)$ with $\Gal(F)\simeq D_4$ and $|\Delta(F)|<X$}\right\},\\[0.5pt]
N_{V_4}'(X) & =\#\left\{[F]: \mbox{irreducible $F\in \displaystyle\bigcup_{i=1}^{3}V_{\J^{(i)}}^{\mmax}(\bZ)$ with $\Gal(F)\simeq V_4$ and $|\Delta(F)|<X$}\right\},\end{align*}
where $[-]$ denotes $\GL_2(\bZ)$-equivalence class. \\

\subsection{Proof of Theorem~\ref{h(f,g)}}

First, analogous to (\ref{action}), we have the so-called \emph{twisted action} of $\GL_2(\bR)$ on the set of real binary quadratic forms, given by
\begin{equation}\label{quadratic action} \varphi_T(x,y) = \frac{1}{\det(T)}\varphi(t_1x + t_2y, t_3x + t_4y)\mbox{ for }T = \begin{pmatrix}t_1&t_2\\t_3&t_4\end{pmatrix}.\end{equation}
This is compatible with (\ref{action}), in the sense that $(\varphi\psi)_T = \varphi_T\psi_T$ for any real binary quadratic forms $\varphi$ and $\psi$. Analogous to the binary quartic form case, we then define
\begin{align*}
W_\J(\bR) & = \{\mbox{real binary quadratic forms $\varphi$ such that $\varphi_{M_\J} = -\varphi$}\},\\\notag
W_\J(\bZ) & = \{\mbox{integral binary quadratic forms $\varphi$ such that $\varphi_{M_\J} = -\varphi$}\}.
\end{align*}
Given a real binary quadratic form $\varphi$, by a direct computation, we have
\begin{equation}\label{W char}
\varphi \in W_\J(\bR) \mbox{ if and only if }2\gamma\varphi_2 - \beta\varphi_1 + 2\alpha\varphi_0 = 0.
\end{equation}
Considerations of the Jacobian determinant have the following consequence:

\begin{proposition}\label{Jfg lemma}Let $(f,g)$ be a pair of non-proportional real binary quadratic forms. Then $\J_{(f,g)}(x,y)$ is proportional to $\J(x,y)$ if and only if $f,g\in W_\J(\bR)$.
\end{proposition}
\begin{proof}By the definition of Jacobian determinant, we have
\[ 2\J_{(f,g),0}f_2 - \J_{(f,g),1}f_1 + 2\J_{(f,g),2}f_0 = 0 = 2\J_{(f,g),0}g_2 - \J_{(f,g),1}g_1 + 2\J_{(f,g),2}g_0,\]
and $f,g$ being non-proportional implies $\J_{(f,g)}(x,y) \neq 0$. Thus, if $\J_{(f,g)}(x,y)$ is proportional to $\J(x,y)$, then $f,g\in W_\J(\bR)$ by (\ref{W char}). Conversely, if $f,g\in W_\J(\bR)$, then by (\ref{W char}) we have
\[ 2\gamma f_2 - \beta f_1 + 2\alpha f_0 = 0 = 2\gamma g_2 - \beta g_1 + 2\alpha g_0.\]
A straightforward computation shows that
\begin{align*}
    \alpha \J_{(f,g),1} = 2\alpha (f_2g_0 - f_0g_2) &= \beta(f_2g_1 - f_1g_2) = \beta\J_{(f,g),2},\\
    \alpha\J_{(f,g),0} = \alpha(f_1g_0 - f_0g_1) & = \gamma (f_2g_1 - f_1g_2) = \gamma \J_{(f,g),2},\\
    \beta\J_{(f,g),0} = \beta(f_1g_0 - f_0g_1) & = 2\gamma(f_2g_0 - _0g_2) = \gamma\J_{(f,g),1},
\end{align*}
which implies that $\J_{(f,g)}(x,y)$ and $\J(x,y)$ are proportional.
\end{proof}

In \cite[(3-1) and (3-2)]{TX}, we observed that
\[\Psi_V:a_4x^4+a_3x^3y+a_2x^2y^2+a_1xy^3+a_0y^4\mapsto\begin{cases}(a_4,a_3,a_2) &\mbox{if $\alpha\ne0$}\\ (a_4,a_2,a_0) &\mbox{if $\alpha=0$}\end{cases}\]
defines an isomorphism of vector spaces from $V_\J(\bR)$ to $\bR^3$, and the image of $V_{\J}(\bZ)$ under this map is a sublattice $\Lambda_\J$ of $\bZ^3$. Similarly, the characterization (\ref{W char}) implies that
\[\Psi_W:\varphi_2x^2 + \varphi_1xy + \varphi_0y^2 \mapsto\begin{cases}
 (\varphi_2,\varphi_1) &\mbox{if $\alpha\neq0$}\\
 (\varphi_2,\varphi_0)&\mbox{if $\alpha=0$}
\end{cases}\]
defines an isomorphism of vector spaces from $W_{\J}(\bR)$ to $\bR^2$, and the image of $W_\J(\bZ)$ under this map is a sublattice $\L_\J$ of $\bZ^2$. The key to the proof of Theorem~\ref{h(f,g)} is the observation that any $\bZ$-basis $(f,g)$ of $W_\J(\bZ)$ may be lifted to a $\bZ$-basis $(f^2,fg,g^2)$ of $V_\J(\bZ)$. We shall prove this by comparing determinants. Put $n_\beta=2$ if $\beta$ is odd, and $n_\beta=1$ if $\beta$ is even.

\begin{lemma}\label{V lemma}We have $\det(\Lambda_\J) = (n_\beta\alpha)^3$ if $\alpha\neq0$, and $\det(\Lambda_\J) = (n_\beta\beta/2)^3$ if $\alpha=0$.
\end{lemma}
\begin{proof}For $\alpha \neq 0$, see \cite[Proposition 3.4]{TX}. For $\alpha=0$, this is easy to verify. 
\end{proof}

\begin{lemma}\label{W lemma}We have $\det(\L_\J) = n_\beta\alpha$ if $\alpha\neq0$, and $\det(\L_\J)=n_\beta\beta/2$ if $\alpha=0$.
\end{lemma}
\begin{proof}By definition and (\ref{W char}), we have
\[ \L_\J =\begin{cases}\{(\varphi_2,\varphi_1)\in\bZ^2: 2\gamma\varphi_2 -\beta\varphi_1 \equiv 0 \mbox{ (mod $2\alpha$)}\}&\mbox{if $\alpha\neq0$},\\
\{(\varphi_2,\varphi_0)\in\bZ^2:  2\gamma \varphi_2 \equiv 0\mbox{ (mod $\beta$})\} &\mbox{if $\alpha=0$},\end{cases}\]
and recall that $\alpha,\beta,\gamma$ are pairwise coprime. We easily verify that
\[\begin{cases}\{(\alpha,0),(0,1)\} &\mbox{if $\alpha\neq0$ and $\beta=0$}\\
\{(n_\beta\beta/2,0),(0,1)\}&\mbox{if $\alpha=0$}
\end{cases}\]
is a $\bZ$-basis for $\L_\J$, so $\det(\L_\J)$ is as stated in these two cases. Next, suppose that $\alpha,\beta\neq0$, and we shall use the fact that
\[ \det(\L_\J) = \prod_{p}\det(\L_{\J,p}),\mbox{ where }\L_{\J,p} = \bZ_p\otimes_\bZ\L_\J.\]
For any prime $p$, write $\alpha = p^ka$ and $\beta = p^\ell b$, where $k,\ell\geq 0$ and $a,b$ are coprime to $p$. Then
\[ \L_{\J,p} = \{(\varphi_2,\varphi_1)\in\bZ_p^2 : 2\gamma\varphi_2 - p^\ell b\varphi_1 \equiv 0\mbox{ (mod $p^{k+\ep_p}$)}\},\]
where $\ep_p = 0$ for $p$ odd and $\ep_2=1$.
It is not hard to verify that
\[\begin{cases}
\{(1,2\gamma b^{-1}),(0,p^{k+\ep_p})\} &\mbox{if $\ell=0$}\\
\{(p^k,0),(0,1)\}&\mbox{if }\ell\geq1\AND\ell\geq k +\ep_p\\
\{(p^{\ell - \ep_p},2p^{-\ep_p}\gamma b^{-1}),(0,p^{k-\ell + \ep_p})\}&\mbox{if }\ell\geq1\AND\ell\leq k+\ep_p
\end{cases}\]
is a $\bZ_p$-basis for $\L_{\J,p}$. Hence, we have $\det(\L_{\J,p})=p^{k}$ for $p$ odd while $\det(\L_{\J,2})=n_\beta2^{k}$. This shows that $\det(\L_\J) = n_\beta\alpha$, as claimed.
\end{proof}

Lemma~\ref{W lemma} implies the following characterization of primitivity:

\begin{proposition}\label{primitive char}For any non-proportional $f,g\in W_\J(\bZ)$, the following are equivalent.
\begin{enumerate}[(i)]
\item The pair $(f,g)$ is primitive.
\item The pair $(f,g)$ is a $\bZ$-basis for $W_\J(\bZ)$.
\item We have $(\J_{(f,g),2},\J_{(f,g),1}) = \pm(n_\beta\alpha, n_\beta\beta)$.
\item We have $\J_{(f,g)}(x,y) = \pm n_\beta \J(x,y)$.
\end{enumerate} 
\end{proposition}
\begin{proof}By Proposition~\ref{Jfg lemma}, we know that $\J_{(f,g)}(x,y)$ and $\J(x,y)$ are proportional. Since $\J(x,y)$ is primitive, we have $\J_{(f,g)}(x,y) = \pm n \J(x,y)$ for some $n\in\bN$. It follows that
\[ \gcd\left(\J_{(f,g),2}, \J_{(f,g),1}/2, \J_{(f,g),0}\right) = \gcd(n\alpha, n\beta/2, n\gamma) = n/n_\beta,\]
and so $(f,g)$ is primitive precisely when $n = n_\beta$. This gives the equivalence between (i) and (iv), and that between (iv) and (iii) is obvious. Let $\L_{(f,g)}$ be the sublattice of $\L_\J$ spanned by $\Phi_W(f)$ and $\Phi_W(g)$ over $\bZ$. Then $(f,g)$ is a $\bZ$-basis of $W_\J(\bZ)$ exactly when $\det(\L_{(f,g)}) = \det(\L_\J)$. But
\[\det(\L_{(f,g)}) = \det\begin{pmatrix}\Psi_W(f) \\ \Psi_W(g) \end{pmatrix} = \begin{cases}
\J_{(f,g),2} & \mbox{if }\alpha \neq 0\\
\J_{(f,g),1}/2 & \mbox{if }\alpha = 0
\end{cases}\]
so the equivalence between (iii) and (ii) now follows from Lemma~\ref{W lemma}.
\end{proof}

\begin{proposition}\label{lift basis}For any $f,g\in W_\J(\bZ)$, if $(f,g)$ is a $\bZ$-basis of $W_\J(\bZ)$, then $(f^2,fg,g^2)$ is a $\bZ$-basis of $V_\J(\bZ)$.
\end{proposition}
\begin{proof}Let $\Lambda_{(f,g)}$ be the sublattice of $\Lambda_\J$ spanned by $\Psi_V(f^2),\Psi_V(fg)$ and $\Psi_V(g^2)$ over $\bZ$. Then $(f^2,fg,g^2)$ is a $\bZ$-basis of $V_\J(\bZ)$ precisely when $\det(\Lambda_{(f,g)}) = \det(\Lambda_\J)$. We have
\[ \det(\Lambda_{(f,g)}) =  \det\begin{pmatrix} \Psi_V(f^2) \\ \Psi_V(fg) \\ \Psi_V(g^2) \end{pmatrix} =\begin{cases} (\J_{(f,g),2})^3 & \mbox{if $\alpha\neq0$}\\ (\J_{(f,g),1}/2)^3 & \mbox{if $\alpha=0$}\end{cases}\]
by a direct calculation. Hence, if $(f,g)$ is a $\bZ$-basis of $W_\J(\bZ)$, then we deduce from  Lemma~\ref{V lemma} and Proposition~\ref{primitive char} (iii) that $\det(\Lambda_{(f,g)}) = \det(\Lambda_\J)$, whence the claim.
\end{proof}

\begin{proof}[Proof of Theorem~\ref{h(f,g)}] Let $F$ be an integral binary quartic form of non-zero discriminant.  First, suppose that $F\in V_{\J}(\bZ)$ and let $(f,g)$ be a $\bZ$-basis of $W_\J(\bZ)$. Then, there is an integral binary quadratic form $h$ such that $F = h(f,g)$  by Proposition~\ref{lift basis}. Notice that  $(f,g)$ is primitive by Proposition~\ref{primitive char} and $\J_{(f,g)}(x,y)$ is proportional to $\J(x,y)$ by Proposition~\ref{Jfg lemma}.\\\\
Conversely, suppose that $F = h(f,g)$ for some integral binary quadratic forms $h,f,g$ such that $\J_{(f,g)}(x,y)$ is proportional to $\J(x,y)$. Since $\Delta(F)\neq0$ implies that $f$ and $g$ are non-proportional, we have $f,g\in W_\J(\bZ)$ by Proposition~\ref{Jfg lemma}. It follows that $F_{M_\J} = F$ and so $F\in V_\J(\bZ)$. \\ \\
Recall the notation from Definition~\ref{con}, and note that the coefficients of $F = h(f,g)$ are given as in (\ref{ai}) below. Using the standard formula for the discriminant of a binary quartic form, we computed in SageMath \cite{Sage} that
\[\label{DiscF} \Delta(F) = \Delta(h(f,g)) = (h_1^2 - 4h_2h_0)^2 H_1^2H_2 = (\Delta(h) H_1)^2 H_2,\]
where $H_2$ is some integer, and
\begin{align*} H_1 & = f_2^2 g_0^2 - f_2 f_1g_1 g_0 + f_2 f_0 g_1^2 + f_1^2 g_2 g_0 - 2 f_2 f_0 g_2 g_0 - f_1 f_0 g_2 g_1 + f_0^2 g_2^2 \\
& = (f_2 g_0 - f_0 g_2)^2 - (f_2 g_1 - f_1 g_2)(f_1 g_0 - f_0 g_1) \end{align*}
is equal to $\Delta(\J_{(f,g)})/4$. This shows that indeed $(\Delta(h)\Delta(\J_{(f,g)})/4)^2$ divides $\Delta(F)$.\\\\
Finally, suppose that $F = h(f,g)$ and $F= h'(f',g')$ are two such representations of $F$ with $(f,g)$ and $(f',g')$ primitive. By Proposition~\ref{primitive char} (ii), there is $T\in \GL_2(\bZ)$ such that $(f',g') =\mbox{}_{T}(f,g)$. But then $h(f,g) = h_T'(f,g)$, and by Proposition~\ref{lift basis} the triplet $(f^2,fg,g^2)$ is linearly independent. Thus, we have $h = h'_T$, whence $h$ and $h'$ are $\GL_2(\bZ)$-equivalent.
\end{proof}

\subsection{Proof of Theorem~\ref{necessary thm'}}

Let $F\in V_\J(\bZ)$ be irreducible and write $F = h(f,g)$ as given by Theorem~\ref{h(f,g)}. For brevity, put
\[ k = \Delta(h),\,\ \J_{(f,g),1}' = \J_{(f,g),1}/2,\AND d = \gcd\left(\J_{(f,g), 2},\J_{(f,g),1}', \J_{(f,g),0}\right).\]
By a direct computation, the coefficients of $F$ are given by
\begin{equation} \label{ai}\begin{cases}a_4= h_2f_2^2 + h_1f_2g_2 + h_0g_2^2,\\
a_3= 2h_2f_2f_1 + h_1(f_2g_1 + f_1g_2) + 2h_0g_2g_1,\\
a_2= h_2(2f_2f_0 + f_1^2) + h_1(f_2g_0 + f_1g_1 + f_0g_2) + h_0(2g_2g_0 + g_1^2),\\
a_1= 2h_2f_1f_0 + h_1(f_1g_0 + f_0g_1) + 2h_0g_1g_0,\\
a_0= h_2f_0^2 + h_1f_0g_0 + h_0g_0^2,\end{cases}\end{equation}
and $F$ may be factored over $\bC$ as
\[\textstyle F = h_2\left(f - \frac{-h_1 + \sqrt{k}}{2h_2}g\right)\left(f - \frac{-h_1 - \sqrt{k}}{2h_2}g\right).\]
Let us write the first factor in parentheses as $Ax^2 + Bxy + Cy^2$, where explicitly
\[\begin{cases}
A = A_1 + A_2\sqrt{k}\\ B = B_1 + B_2\sqrt{k},\\C = C_1 + C_2\sqrt{k},
\end{cases}\AND\hspace{1em}
\begin{cases}
A_1 = \frac{2h_2f_2 + h_1g_2}{2h_2}\mbox{ and }A_2 = \frac{-g_2}{2h_2},\\[1.5ex]
B_1 = \frac{2h_2f_1 + h_1g_1}{2h_2}\mbox{ and }B_2 = \frac{-g_1}{2h_2},\\[1.5ex]
C_1 = \frac{2h_2f_0 + h_1g_0}{2h_2}\mbox{ and }C_2 = \frac{-g_0}{2h_2}.
\end{cases}\]
Without loss of generality, take $\theta_F = (-B + \sqrt{B^2 - 4AC})/(2A)$ and put $L_F = \bQ(\theta_F)$.\\

Consider an element in $Q_F$, say
\begin{equation}\label{element x} x = c_0 + c_1(a_4\theta_F) + c_2(a_4\theta_F^2+a_3\theta_F) + c_3(a_4\theta_F^3 + a_3\theta_F^2 + a_2\theta_F),\end{equation}
where $c_0,c_1,c_2,c_3\in\bZ$. We shall write
\[ x - c_0= x_1 + x_2 \sqrt{B^2 - 4AC}\AND x_2 = \frac{1}{2A^3}\left(y_1 + y_2\sqrt{k}\right),\]
where $x_1,x_2\in K$ and $y_1,y_2\in\bQ$. Explicitly, we have
\begin{align*}
x_1 & = \frac{-B}{2A}(c_1a_4 + c_2a_3 + c_3a_2) + \frac{B^2 - 2AC}{2A^2}(c_2a_4 + c_3a_3) + \frac{B(-B^2 + 3AC)}{2A^3}(c_3a_4),\\[0.5ex]
x_2 & =\frac{1}{2A}(c_1a_4 + c_2a_3 + c_3a_2) - \frac{B}{2A^2}(c_2a_4 + c_3a_3) + \frac{B^2-AC}{2A^3}(c_3a_4),
\end{align*}
and also
\begin{align*} y_1 &=(A_1^2 + kA_2^2)(c_1a_4 + c_2a_3 + c_3a_2)\\ & \hspace{0.5cm}- (A_1B_1 + kA_2B_2)(c_2a_4 + c_3a_3) + (B_1^2 + kB_2^2- A_1C_1 - kA_2C_2)(c_3a_4),\\
y_2& =  (2A_1A_2)(c_1a_4 + c_2a_3 + c_3a_2) \\&\hspace{0.5cm} - (A_1B_2 + A_2B_1)(c_2a_4 + c_3a_3)+ (2B_1B_2 - A_1C_2 - A_2C_1)(c_3a_4).\end{align*}
Now, we know that $Q_F$ is an order in $L$ by \cite{Nakagawa}. Clearly $K = \bQ(\sqrt{k})$ is a quadratic subfield of $L$. Let $\mathcal{O}_L$ and $\mathcal{O}_K$, respectively, denote the rings of integers in $L$ and $K$. Then, we have
\begin{equation}\label{maximal iff} \left(Q_F\mbox{ is maximal}\right)\iff \left(Q_F = \mathcal{O}_L\right) \implies \left(Q_F\cap K = \mathcal{O}_K\right).\end{equation}
Below, we shall determine an explicit $\bZ$-basis of $Q_F\cap K$, namely 
\begin{equation}\label{QF cap K}  Q_F\cap K = \bZ\oplus\bZ\tau,\mbox{ where }\tau=\frac{\Delta(\J_{(f,g)})}{4d}\cdot\frac{-h_1 + \sqrt{k}}{2},\end{equation}
which then allows us to prove Theorem~\ref{necessary thm'}. The key is the next lemma:

\begin{lemma}\label{in K lem}Let $x\in Q_F$ be as in (\ref{element x}). Then, we have $x\in K$ if and only if 
\begin{equation}\label{in K relation}c_2\J_{(f,g),2} + c_3\J_{(f,g),1}' =  - c_1\J_{(f,g),2} + c_3\J_{(f,g),0} = c_1\J_{(f,g),1}' + c_2\J_{(f,g),0} =  0.\end{equation}
Moreover, in this case, we have
\[x-c_0= \begin{cases}\dfrac{c_3}{\J_{(f,g),2}}\left((s_1\J_{(f,g),1}' - s_2\J_{(f,g),2}) - \dfrac{\Delta(\J_{(f,g)})}{4}\left(\dfrac{-h_1 + \sqrt{k}}{2}\right)\right) & \mbox{if $\J_{(f,g),2}\neq0$},\\[2.75ex]
\dfrac{-c_2}{\J_{(f,g),1}'}\left((s_1\J_{(f,g),1}' - s_2\J_{(f,g),2}) - \dfrac{\Delta(\J_{(f,g)})}{4}\left(\dfrac{-h_1 + \sqrt{k}}{2}\right)\right)&\mbox{if $\J_{(f,g),1}'\neq0$},
\end{cases}\]
where $s_1,s_2\in\bZ$, and in particular $x\in \bZ\oplus\bZ\tau$.
\end{lemma}
\begin{proof}Clearly $x\in K$ if and only if $x_2=0$, which is equivalent to $y_1 = y_2 = 0$. We may define 
\[\begin{cases}
\eta_{11} = 2A_1A_2a_4,\\
\eta_{12} = (A_1^2 + kA_2^2)a_4,\\
\eta_{21} = 2A_1A_2a_3 - (A_1B_2 + A_2B_1)a_4,\\
\eta_{22} = (A_1^2 + kA_2^2)a_3 - (A_1B_1 + kA_2B_2)a_4,\\
\eta_{31} = 2A_1A_2a_2 - (A_1B_2 + A_2B_1)a_3 + (2B_1B_2 - A_1C_2 - A_2C_1)a_4,\\
\eta_{32} = (A_1^2 + kA_2^2)a_2 - (A_1B_1 + kA_2B_2)a_3 + (B_1^2 + kB_2^2 - A_1C_1 - kA_2C_2)a_4,
\end{cases}\]
so that $c_i$ is eliminated in $\eta_{i1} y_1 - \eta_{i2}y_2$ for each $i=1,2,3$. 
In SageMath \cite{Sage}, we computed that
\[\begin{bmatrix} \eta_{11} & -\eta_{12} \\ \eta_{21} & -\eta_{22} \\ \eta_{31} & -\eta_{32}\end{bmatrix}\begin{bmatrix} y_1 \\ y_2\end{bmatrix} 
= -\frac{a_4^3}{2h_2^2}\begin{bmatrix}\hspace{3mm}c_2\J_{(f,g),2} + c_3\J_{(f,g),1}'\\ -c_1\J_{(f,g),2}  + c_3\J_{(f,g),0} \\-c_1\J_{(f,g),1}' - c_2\J_{(f,g),0} \end{bmatrix} \AND \begin{cases} \eta_{21}\eta_{12} - \eta_{11}\eta_{22} = \frac{a_4^3}{2h_2^2}\J_{(f,g),2},\\\eta_{31}\eta_{12} - \eta_{11}\eta_{32}=\frac{a_4^3}{2h_2^2}\J_{(f,g),1}'.\end{cases}\]
The two expressions on the right cannot be both zero because $\Delta(F),\Delta(\J_{(f,g)})\neq0$. Hence, the above $(3\times2)$-matrix has trivial null space, so $y_1 = y_2 = 0$ if and only if (\ref{in K relation}) holds.\\

Now, suppose that $x\in K$, which implies $x_2=0$. We  then deduce that
\[ x - c_0 = x_1 + Bx_2 = \frac{-C}{A}(c_2a_4 + a_3a_3) + \frac{BC}{A^2}(c_3a_4) =\frac{1}{(A_1^2 - kA_2^2)^2}(z_1 + z_2 \sqrt{k}),\]
where $z_1,z_2\in\bQ$ are given by
\begin{align*} z_1 &= -(A_1^2 - kA_2^2)(A_1C_1 - kA_2C_2)(c_2a_4  + c_3a_3) \\ & \hspace{0.5cm}+ ((A_1^2 + kA_2^2)(B_1C_1+kB_2C_2)- 2kA_1A_2(B_1C_2 + B_2C_1))(c_3a_4),\\
z_2 & = -(A_1^2 - kA_2^2)(A_1C_2 -A_2C_1)(c_2a_4  + c_3a_3) \\&\hspace{0.5cm}+ ((A_1^2 + kA_2^2)(B_1C_2 + B_2C_1) - 2A_1A_2(B_1C_1 + kB_2C_2))(c_3a_4).
\end{align*}
Again, using SageMath \cite{Sage} we check that $A_1^2 - kA_2^2 = a_4/h_2$ and
\begin{align}\label{z1z2}
(h_2/a_4)^2z_1 & = -(c_2(2h_2f_2f_0 + h_1(f_2g_0 + f_0g_2) + 2h_0g_2g_0) \\\notag & \hspace{2.5cm}+ c_3(2h_2f_1f_0 + h_1(f_1g_0 + f_0g_1) + 2h_0g_1g_0))/2,\\\notag
(h_2/a_4)^2z_2 &= (c_2(f_2g_0 - f_0g_2) + c_3(f_1g_0 - f_0g_1))/2.
\end{align}
By (\ref{in K relation}), we know that 
\begin{equation} \label{c2c3} \begin{cases}c_2 = -c_3\J_{(f,g),1}'/\J_{(f,g),2}\AND c_1 = c_3\J_{(f,g),0}/\J_{(f,g),2}&\mbox{if }\J_{(f,g),2}\neq0,\\ 
c_3 = -c_2\J_{(f,g),2}/\J_{(f,g),1}'\AND c_1 = -c_2\J_{(f,g),0}/\J_{(f,g),1}'&\mbox{if }\J_{(f,g),1}'\neq0.
\end{cases}\end{equation}
By substituting (\ref{c2c3}) into (\ref{z1z2}), we obtain the desired expression for $x-c_0$ with
\[s_1 = h_2f_2f_0 + h_1f_0g_2 + h_0g_2 g_0\AND s_2 = h_2f_1f_0 + h_1f_0g_1 + h_0g_1g_0.\]
The above relations also imply that $\J_{(f,g),2}/d$ divides $c_3$ if $\J_{(f,g),2}\neq0$, and that $\J_{(f,g),1}'/d$ divides $c_2$ if $\J_{(f,g),1}'\neq0$. We then see that $x\in\bZ\oplus\bZ\tau$, as claimed.
\end{proof}

\begin{proof}[Proof of Theorem~\ref{necessary thm'}]First, we shall prove claim (\ref{QF cap K}). By Lemma~\ref{in K lem}, it remains to show the inclusion $\bZ\oplus\bZ\tau\subset Q_F\cap K$, and it suffices to show that $\tau\in Q_F\cap K$. Put
\[ c_0 = (s_1\J_{(f,g),1}' - s_2\J_{(f,g),2})/d,\,\ c_1 = -\J_{(f,g),0}/d,\,\ c_2 =  \J_{(f,g),1}'/d,\,\ c_3 = -\J_{(f,g),2}/d,\]
where $s_1,s_2\in\bZ$ are as above. Clearly $c_0,c_1,c_2,c_3\in\bZ$ satisfy (\ref{in K relation}). From Lemma~\ref{in K lem}, we then deduce that the element in (\ref{element x}) lies in $K$ and is equal to $\tau$, whence the claim.\\\\
Now, write $k = m^2k'$ for $m\in\bN$ and $k'\in\bZ$ square-free. Notice that $m\neq 1$ if $k'\not\equiv1\pmod{4}$ because $k$ is a discriminant. We know that
\[\mbox{disc}(\mathcal{O}_K) = \begin{cases}4k'&\mbox{if $k'\not\equiv 1$ (mod $4$)},\\ k'
&\mbox{if $k'\equiv1$ (mod $4$)},\end{cases} \AND \mbox{disc}(Q_F\cap K) = \left(\frac{\Delta(\J_{(f,g)})}{4d}\right)^2m^2k'\]
by (\ref{QF cap K}). Suppose that $Q_F$ is maximal. Then $\mbox{disc}(\mathcal{O}_K)$ and $\mbox{disc}(Q_F\cap K)$ are equal by (\ref{maximal iff}), and we deduce that $k$ is a fundamental discriminant and $\Delta(\J_{(f,g)}) = \pm 4d$. Since $\Delta(\J_{(f,g)})$ is divisible by $4d^2$ by definition, we must in fact have $d = 1$, whence $(f,g)$ is primitive. By Proposition~\ref{primitive char}, this in turn implies that $\J_{(f,g)} = \pm 2\cdot\J$ if $\beta$ is odd, and $\J_{(f,g)}  = \pm \J$ if $\beta$ is even. It follows that $\Delta(\J) \in \{1,4,-4\}$, as claimed.
\end{proof}

\subsection{The conductor polynomial}\label{conductor poly}

Let us make the following definition:

\begin{definition}\label{conductor def}For any $F\in V_\J(\bZ)$ of non-zero discriminant, write $F = h(f,g)$ as in Theorem~\ref{h(f,g)} with $(f,g)$ primitive, and define the \emph{conductor polynomial of $F$ with respect to $\J$} by $\C_\J(F) = \Delta(F)/\Delta(h)$. This does not depend on the choice of $h,f,g$. For convenience, we shall define $\C_\J(F) = 0$ when $\Delta(F)=0$.
\end{definition}

The conductor polynomial $\C_\J(\cdot)$ is an invariant in a certain sense, as the following proposition demonstrates:

\begin{proposition}\label{invariant}For any $F\in V_{\J}(\bZ)$ and $T\in\GL_2(\bZ)$, we have $\C_{\J_T}(F_T) = \C_\J(F)$.
\end{proposition}
\begin{proof}
Observe that any $T\in \GL_2(\bZ)$ induces an isomorphism
\[ W_\J(\bZ) \longrightarrow W_{\J_T}(\bZ);\hspace{1em} \varphi\mapsto \varphi_T\]
of $\bZ$-modules. Let us write $F=h(f,g)$ as in Theorem \ref{h(f,g)} with $(f,g)$ primitive and $\J_{(f,g)}(x,y)$ proportional to $\J(x,y)$. By Proposition \ref{Jfg lemma}, we know that $f,g\in W_\J(\bZ)$, and $\J_{(f_T,g_T)}(x,y)$ is proportional to $\J_T(x,y)$. By Proposition \ref{primitive char}, the pair $(f,g)$ is a $\bZ$-basis for $W_{\J}(\bZ)$, so clearly $(f_T,g_T)$ is a $\bZ$-basis for $W_{\J_T}(\bZ)$. This implies that $(f_T,g_T)$ is primitive, again by Proposition \ref{primitive char}. It then follows from Definition \ref{conductor def} that
\[ \C_{\J_T}(F_T) = \C_{\J_T}(h(f_T,g_T))= \Delta(F_T)/\Delta(h) = \Delta(F)/\Delta(h) = \C_\J(F),\]
as claimed.
\end{proof}

Given an integral and irreducible binary quartic form $F$ with $\Gal(F)\simeq D_4,C_4$, Proposition~\ref{criterion} implies that $F\in V_\J(\bZ)$ for a unique choice of $\J$ up to scaling. We may then write $\C(F)$ without the subscript $\J$, and refer to it as the \emph{conductor polynomial of $F$}. This is related to the conductor of quartic $D_4$-fields as follows:

\begin{proposition}\label{conductor vs}For any integral and irreducible binary quartic form $F$ such that $Q_F$ is maximal, we have $|\Delta(F)| = |\operatorname{disc}(L_F)|$, as well as $|\C(F)| = |\C(L_F)|$ when $\Gal(F)\simeq D_4$.
\end{proposition}
\begin{proof}By \cite[Proposition 1.1]{Nakagawa}, we know that $\Delta(F) = \mbox{disc}(Q_F)$, and so the first claim holds by (\ref{maximal iff}). Write $F = h(f,g)$ as in Theorem~\ref{h(f,g)} with $(f,g)$ primitive. Notice that $\bQ(\sqrt{\Delta(h)})$ is a quadratic subfield of $L_F$ and $\Delta(h)$ is a fundamental discriminant by Theorem~\ref{necessary thm'}. We then see that the second claim holds as well.
\end{proof}

Finally, we note that the terminology \emph{conductor polynomial} comes from the fact that $\C_\J(F)$ is a polynomial in the coefficients of $F$. More specifically, let $F\in V_\J(\bZ)$ be as in (\ref{generic F}) of non-zero discriminant. By \cite[(3-1) and (3-2)]{TX}, we know that if $\alpha\neq 0$ and $\beta,\beta^2+4\alpha\gamma\neq0$, respectively, then the coefficients $(a_1,a_0)$ and $(a_3,a_1)$ are determined by the remaining three, and it may be computed in SageMath \cite{Sage} for example, that
\begin{align*}
    \Delta(F) & = \frac{(\beta^2-4\alpha\gamma)^2}{64\alpha^9}\cdot P_\J(a_4,a_3,a_2)\cdot (-8\gamma a_4^2+\alpha a_3^2 - 4\alpha a_4a_2 + 2\beta a_4a_3)^2,\\
    \Delta(F) & = \frac{16(\beta^2-4\alpha\gamma)^2}{\beta^6(\beta^2+4\alpha\gamma)^6}\cdot P_\J'(a_4,a_2,a_0)\cdot(16\gamma^4 a_4^2 + \beta^4a_2^2 + 16\alpha^4a_0^2 - 8\beta^2\gamma^2 a_4a_2 \\
&\hspace{4cm}  - 4\beta^4a_4a_0 - 32\alpha\beta^2\gamma a_4a_0 - 32\alpha^2\gamma^2 a_4a_0 - 8\alpha^2\beta^2a_2a_0)^2,
\end{align*}
where $P_\J$ and $P_\J'$ are some polynomials with coefficients in $\bZ[\alpha,\beta,\gamma]$. Now, write $F = h(f,g)$ with $(f,g)$ primitive given by Theorem \ref{h(f,g)}, and $\J_{(f,g)}(x,y) = \pm n_\beta \J(x,y)$ by Proposition \ref{primitive char}. Using this, together with the coefficients of $F$ given explicitly in (\ref{ai}), again we computed in SageMath \cite{Sage} that if $\alpha\neq0$ and $\beta,\beta^2+4\alpha\gamma\neq0$, respectively, then
\begin{align*} \Delta(h) & = \frac{1}{n_\beta^2\alpha^3}\cdot(- 8\gamma a_4^2+\alpha a_3^2 - 4\alpha a_4a_2 + 2\beta a_4a_3),\\
\Delta(h) & = \frac{4}{n_\beta^2\beta^2(\beta^2+4\alpha\gamma)^2}\cdot(16\gamma^4 a_4^2 + \beta^4a_2^2 + 16\alpha^4a_0^2 - 8\beta^2\gamma^2 a_4a_2 \\
&\hspace{2.75cm}  - 4\beta^4a_4a_0 - 32\alpha\beta^2\gamma a_4a_0 - 32\alpha^2\gamma^2 a_4a_0 - 8\alpha^2\beta^2a_2a_0).\end{align*} 
Regarding $\alpha,\beta$ and $\gamma$ as fixed constants, we then see that indeed $\C_\J(F)$ may be expressed as a polynomial in three of the coefficients of $F$.

\section{The number of pre-maximal binary quartic forms having bounded conductor}
\label{pre-max}

In Section \ref{reduce section}, we showed that integral binary quartic forms corresponding to quartic $D_4$-fields with monogenic cubic resolvent may be $\GL_2(\bZ)$-translated into one of the three families $V_{\J^{(i)}}(\bZ)$ for $i = 1,2,3$. In this section, we shall count these so-called pre-maximal forms
\[\label{VjiZ} 
V_{\J^{(i)}}(\bZ)(X)  = \{F \in V_{\J^{(i)}}(\bZ) : 0 < |\C_{\J^{(i)}}(F)|< X \}\]
by their conductor, without imposing any maximality condition.\\

Let us make two observations. First, for any $F\in V_{\J^{(i)}}(\bZ)$ as in (\ref{F in J}), it is easy to check that 
\begin{align}\label{Ci}
\C_{\J^{(i)}}(F) & = \Delta(F)/(B^2-4AC) \\\notag
& = \begin{cases} 16AC(B^2 - 4AC) & \mbox{for $i=1$},\\
(4A-2B +C)(4A+2B+C)(B^2-4AC)&\mbox{for $i=2$},\\
((4A-C)^2 + 4B^2)(B^2 -4AC)&\mbox{for $i=3$}.
\end{cases}\end{align}
Second, the two families $V_{\J^{(1)}}(\bZ)$ and $V_{\J^{(2)}}(\bZ)$ exhibit an obvious relation because
\[ M_{\J^{(1)}} = \frac{1}{2}\left(T_0^{-1}M_{\J^{(2)}} T_0\right)\mbox{ for }T_0 = \begin{pmatrix} 1 & 1 \\ -1 & 1 \end{pmatrix}\mbox{ where }M_{\J^{(1)}} = \begin{pmatrix} 1 & 0 \\ 0 & -1 \end{pmatrix},\, M_{\J^{(2)}} = \begin{pmatrix} 0 & -2 \\ -2 & 0 \end{pmatrix}\] 
are as in (\ref{MJ def}). In particular, there is an isomorphism of $\bR$-vector spaces defined by 
\begin{equation}\label{map 2 to 1}\iota: V_{\J^{(2)}}(\bR) \longrightarrow V_{\J^{(1)}}(\bR);\hspace{1em}\iota(F) =  \det(T_0)^2\cdot F_{T_0} = 4\cdot F_{T_0},\end{equation}
where the factor $\det(T_0)^2$ is simply to ensure that $\iota$ preserves integrality, because we are using the twisted action (\ref{action}). Identifying $V_{\J^{(1)}}(\bR)$ and $V_{\J^{(2)}}(\bR)$ with $\bR^3$ via (\ref{F in J}), explicitly
\begin{equation}\label{iota explicit}
\iota(A,B,C) = (4A-2B+C, 2(4A-C), 4A+2B+C).\end{equation}
Also, a direct calculation yields
\[ \begin{pmatrix} 4 & -2 & 1 \\ 8 & 0 & -2 \\ 4 & 2 & 1\end{pmatrix}^{-1} = \frac{1}{16}\begin{pmatrix} 1 & 1 & 1  \\ -4 & 0 & 4 \\ 4 & -4 & 4 \end{pmatrix} \AND \det\begin{pmatrix} 4 & -2 & 1 \\ 8 & 0 & -2 \\ 4 & 2 & 1\end{pmatrix} = 64.\]
We then see that the image of $V_{\J^{(2)}}(\bZ)$ under $\iota$ is the sublattice
\[ \Lambda_\iota = \{(a,b,c)\in\bZ^3:a+b+c\equiv0\ppmod{16},\, -a+c,a-b+c\equiv0\ppmod{4}\}\]
having determinant $64$. The above discussion implies that the enumerations of $V_{\J^{(1)}}(\bZ)$ and $V_{\J^{(2)}}(\bZ)$ are essentially the same problem. We shall then focus on the family $V_{\J^{(1)}}(\bZ)$, and we shall also show that the contribution of $V_{\J^{(3)}}(\bZ)$ is in fact negligible.

\subsection{Planar regions corresponding to bounded conductor} \label{planar section}
As we can see from (\ref{F in J}), for each $i=1,2,3$, the set $V_{\J^{(i)}}(\bZ)$ can be identified with $\bZ^3$, with a corresponding embedding in $\bR^3$. The condition that $|\C_{\J^{(i)}}(F)|$ being bounded therefore determines a certain region in $\bR^3$ of positive measure. Our plan is to exchange these somewhat awkward 3-dimensional regions for nicer 2-dimensional regions. \\

For any $F\in V_{\J^{(i)}}(\bZ)$ given as in (\ref{F in J}), let us put
\[\label{xy definition} (x_i(F),u_i(F),v_i(F)) = \begin{cases}(B,A,C) &\mbox{for $i=1$},\\
(4A-C, 4A-2B+C, 4A+2B+C) &\mbox{for $i=2$},\\
(4A+C,4A-C, 2B) &\mbox{for $i=3$}.
\end{cases}\]
Then, by further defining
\[ y_i(F) = \begin{cases}
4u_i(F)v_i(F) & \mbox{for }i=1,\\
u_i(F)v_i(F) & \mbox{for }i=2,\\ u_i(F)^2 + v_i(F)^2 &\mbox{for }i=3,
\end{cases}\]
we obtain from (\ref{Ci}) the formulae
\begin{align}\label{Ci xy formula}
\C_{\J^{(1)}}(F) & = \Delta(F)/(x_1(F)^2-y_1(F)) = 4y_1(F)(x_1(F)^2 - y_1(F)),\\\notag
\C_{\J^{(2)}}(F) & = 4\Delta(F)/(x_2(F)^2-y_2(F)) = y_2(F)(x_2(F)^2-y_2(F))/4, \\\notag
\C_{\J^{(3)}}(F) & = -4\Delta(F)/(x_3(F)^2-y_3(F)) = -y_3(F)(x_3(F)^2-y_3(F))/4.\end{align}
From the above change of variables, we see that the bound on the conductor polynomial gives the $2$-dimensional region
\[R(X) = \{(x,y) \in \bR^2 : |y|, |x^2 - y| \geq 1,\, |y(x^2 - y)| < X\}.\]
From a point $(x,y)\in R(X)\cap \bZ^2$, for $i=1,2$ we can recover elements in $V_{\J^{(i)}}(\bZ)$ by restricting to appropriate congruence classes and then taking divisors of $y$. But for $i=3$ we must require that $y$ is a sum of two squares. More precisely:

\begin{lemma} \label{RX bijections} We have the equalities
\begin{align*}
 \# V_{\J^{(1)}}(\bZ)(X) &= \sum_{\substack{(x,y) \in R(X/4) \cap \bZ^2 \\ y \equiv 0 \ppmod{4} }} \sum_{uv = y/4} 1,\\
 \# V_{\J^{(2)}}(\bZ)(X) & = \sum_{\substack{(x,y) \in R(4X) \cap \bZ^2  }} \sum_{\substack{uv = y \\ (x,u,v) \in \Lambda_2}} 1,\\
\# V_{\J^{(3)}}(\bZ)(X) & = \sum_{(x,y) \in R(4X) \cap \bZ^2} \sum_{\substack{u^2 + v^2 = y \\ (x,u,v) \in \Lambda_3}}1,
 \end{align*}
where $\Lambda_2$ and $\Lambda_3$ are the sublattices in $\bZ^3$ defined by
\begin{align*}
\Lambda_2 &  = \{(x,u,v)\in\bZ^3 : 2x + u +v \equiv0\ppmod{16},\,-u+v,-2x+u+v\equiv0\ppmod{4}\},\\
\Lambda_3 & = \{(x,u,v)\in\bZ^3 : x+u\equiv0\ppmod{8},\, v,x-u\equiv0\ppmod{2} \}.
\end{align*}
\end{lemma}
\begin{proof}For $i=2,3$, the change of variables from $(x,u,v)$ to $(A,B,C)$ is given by
\[ \begin{pmatrix}
4 & 0 & -1 \\ 4 & -2 & 1 \\ 4 & 2 & 1 
\end{pmatrix}^{-1} = \frac{1}{16}\begin{pmatrix} 2 & 1 & 1 \\ 0 & -4 & 4 \\ -8 & 4 & 4\end{pmatrix}\AND 
\begin{pmatrix}4 & 0 & 1 \\ 4 & 0 & -1 \\  0& 2 & 0
\end{pmatrix}^{-1} = \frac{1}{8}\begin{pmatrix} 1 & 1 & 0 \\ 0 & 0 & 4 \\ 4 & -4 & 0\end{pmatrix},\]
respectively, explaining the definitions of $\Lambda_2$ and $\Lambda_3$.
\end{proof}

The main tool we shall use to convert geometric information to an arithmetic count is the following well-known result of Davenport \cite{Dav}, considered the bedrock of geometry of numbers. We shall use the formulation due to Bhargava \cite{Bha} and Bhargava-Shankar \cite{BS}:

\begin{proposition}[Davenport's lemma] \label{Davenport}Let $\R$ be a bounded, semi-algebraic multi-set in $\bR^n$ having maximum multiplicity $m$ and that is defined by at most $k$ polynomial inequalities, each having degree at most $\ell$. Then the number of integral lattice points (counted with multiplicity) contained in the region $\R$ is given by
\[\Vol(\R) + O(\max\{\Vol(\ol{\R}), 1\}),\]
where $\Vol(\ol{\R})$ denotes the greatest $d$-dimensional volume among the projections of $\R$ onto a coordinate subspace by equating $n-d$ coordinates to zero, where $1\leq d\leq n-1$ The implied constant in the second summand depends only on $n,m,k,\ell$. 
\end{proposition}


\subsection{Counting elements in the three families by conductor} \label{pre-count section}

By Lemma \ref{RX bijections}, the enumerations of $V_{\J^{(i)}}(\bZ)$  may be reduced to counting integral points in $R(X)$ with weights. Let us record the following observations, both of which are easily verified, about the geometry of $R(X)$:

\begin{lemma}\label{bound on y}Let $(x,y)\in\bR^2$.
\begin{enumerate}[(a)]
\item For $|x|\geq \sqrt{2}X^{1/4}$, the condition $|y(x^2-y)|<X$ is equivalent to
\[ y\in \left(\frac{x^2 - \sqrt{x^4 + 4X}}{2},\frac{x^2 - \sqrt{x^4-4X}}{2} \right)\sqcup \left(\frac{x^2 + \sqrt{x^4 - 4X}}{2}, \frac{x^2 + \sqrt{x^4+4X}}{2}\right).\]
\item For $|x| < \sqrt{2}X^{1/4}$, the condition $|y(x^2-y)|<X$ is equivalent to
\[ y\in  \left(\dfrac{x^2 - \sqrt{x^4 + 4X}}{2}, \dfrac{x^2 + \sqrt{x^4+4X}}{2}\right).\]
\end{enumerate}
\end{lemma}

We next show that the region $R(X)$ does not contain any integral points with $x^2 > 2X$. This step is crucial: it shows that the cusp in $R(X)$ is not too large. 

\begin{lemma} \label{x bd} For all $(x,y) \in R(X) \cap \bZ^2$, we have $x^2 \leq 2X$. 
\end{lemma}

\begin{proof}Suppose that $(x,y)\in R(X)\cap\bZ^2$ but $x^2 > 2X$. Clearly $|y|\leq X$, and so $x^2 - y >0$. Since $y$ and $x^2 -y$ sum to $x^2$, one of them has to exceed $X$ in size. But then $|y(x^2-y)|$ is at least $X$, which contradicts that $(x,y)\in R(X)$.
\end{proof}

We shall prove:

\begin{proposition} \label{N1X count} For $i=1,2$, let $N_i^{(r_2)}(X)$ denote the number of forms $F\in V_{\J^{(i)}}(\bZ)(X)$ such that $F(x,1)$ has exactly $4-2r_2$ real roots and $|v_i(F)| > |u_i(F)|$. Then
\[\label{N1X} N_i^{(r_2)}(X) = \frac{\mathfrak{r}(r_2)\sqrt{2}\Gamma(1/4)^2}{96\sqrt{\pi}}X^{3/4} \log X + O\left(X^{3/4} (\log X)^{3/4}\right),
\]
where $\fr(0) = 1$, $\fr(1)=\sqrt{2}$, and $\fr(2) = 1+\sqrt{2}$ are as in Theorem \ref{real split MT}.
\end{proposition}

We remark that the extra requirement $|v_i(F)| > |u_i(F)|$ is to ensure that we only count each $\GL_2(\bZ)$-equivalence class once; see Lemma \ref{representative cor} below.\\

Recall the isomorphism $\iota$ in (\ref{map 2 to 1}) and the definition of $\Lambda_\iota$ after it. By (\ref{iota explicit}) and (\ref{Ci xy formula}), we have
\[ u_1(\iota(F)) = u_2(F),\, v_1(\iota(F)) = v_2(F),\, \C_{\J^{(1)}}(\iota(F)) = 256\cdot \C_{\J^{(2)}}(F)\mbox{ for all }F\in V_{\J^{(2)}}(\bZ).\]
It then follows that
\[ N_2^{(r_2)}(X) = \frac{1}{\det(\Lambda_\iota)}\cdot N_1^{(r_2)}(256X) = \frac{1}{64}\cdot N_1^{(r_2)}(256X),\]
Since the main term has order $X^{3/4}\log X$ and $256^{3/4} = 64$, the case $i=2$ would follow once we prove the case $i=1$.\\

Now, to prove the claim for $i=1$, the first step is to show that the cusps of $R(X)$ with $|x|$ large contribute to an error term:

\begin{lemma} \label{tail cut} We have
\[ \sum_{\substack{(x,y) \in R(X) \cap \bZ^2 \\ |x| \geq X^{1/4}(\log X)^{1/4} }} \sum_{uv = y}1 = O\left(X^{3/4}(\log X)^{3/4}\right) .\]
\end{lemma}

\begin{proof} For each $(x,y)\in R(X)\cap \bZ^2$ with $x \geq X^{1/4}(\log X)^{1/4}$, by Lemma~\ref{bound on y} (a), we have the constraint that $y$ lies in $J_{x,1}\sqcup J_{x,2}$, where
\[ J_{x,1} = \left(\frac{x^2 - \sqrt{x^4 + 4X}}{2}, \frac{x^2 - \sqrt{x^4 - 4X}}{2} \right) \AND J_{x,2} \left(\frac{x^2 + \sqrt{x^4 - 4X}}{2}, \frac{x^2 + \sqrt{x^4 + 4X}}{2} \right).\]
 Both $J_{x,1},J_{x,2}$ have lengths $O \left(X/x^2 \right)$, which is bounded away from zero because $x^2\leq 2X$ by Lemma~\ref{x bd}. Now, we need to evaluate the divisor function $d(\cdot)$ at $y$. What is problematic is that when $|x|$ is close to $X^{1/2}$, then $|y|$ is close to $X$ but lies in an interval of essentially bounded length. To overcome this issue, we shall further subdivide $x$ into the ranges
\[\label{mid range} X^{1/4} (\log X)^{1/4} < |x| < X^{5/16} \AND X^{5/16} \leq |x| \leq \sqrt{2}X^{1/2}.\]
For the first range of $x$, we shall draw upon the following estimate, due to Huxley \cite{Hux}, of the average order of the divisor function over short intervals:
\begin{equation}\label{dr growth}
\sum_{Y \leq y < Y + Y^\theta} d(y) \sim Y^\theta \log Y, \text{ valid for } \theta > 131/416.\end{equation}
We note that the weaker lower bound $\theta > 1/2$ suffices for us, which is obtained from Dirichlet's hyperbola method. The endpoints of $J_{x,1},J_{x,2}$ are of sizes $O(x^2)$. Since $J_{x,1},J_{x,2}$ have lengths $O(X/x^2)$, and the bound $|x|\leq X^{5/16}$ implies $X/x^2 > x^{2\theta}$ for $\theta < 3/5$, it follows that
\[\sum_{y \in J} d(y) \sim \frac{X}{x^2}\cdot \log X \mbox{ for both }J=J_{x,1},J_{x,2}.\]
The total contribution from the range $|x| \leq X^{5/16}$ is then
\begin{equation} \label{mid range cont}X \log X \sum_{X^{1/4} (\log X)^{1/4} < |x| < X^{5/16}} x^{-2} = O \left(X^{3/4} (\log X)^{3/4} \right).\end{equation}
When $X^{5/16} < |x| \leq \sqrt{2} X^{1/2}$, we no longer average the divisor function but use the uniform upper bound $d(y) = O_\ep \left(y^\ep \right)$ for any $\ep > 0$ instead. We have $y = O(X)$ and recall that $y$ lies in two intervals of lengths $O(X/x^2)$. The contribution from the range $|x|\geq X^{5/16}$ is hence
\begin{equation} \label{big range cont} O_\ep \left( X^{1+\ep}\sum_{X^{5/16} < |x| \leq \sqrt{2}X^{1/2}} x^{-2} \right) = O_\ep \left(X^{11/16 + \ep} \right).\end{equation}
Choosing $\ep < 1/16$, we then see that this is an acceptable error term. This completes the proof of the lemma. \end{proof} 

Lemma \ref{tail cut} allows us to eliminate the cusps of $R(X)$ and consider the following nicer subset which is free of extraordinarily long and narrow regions:
\[\label{RofX} \R(X) = \{(x,y) \in \bR^2 : |y|, |x^2 - y| \geq 1,\, |y(x^2 - y)| < X,\, |x| < 2X^{1/4} (\log X)^{1/4} \}.\]
The factor of $2$ in the last inequality is only for convenience when we calculate the area of $\R(X)$ later. Now, we are not simply counting integral points $(x,y)\in\R(X)$ but rather weighing each point by the divisor function evaluated at $y$. For each $a\in\bZ$, upon writing $y = 4av$ we introduce
\begin{align}\label{Ra definition}
    \R_a(X) &= \{(x,v) \in \bR^2 : |v| > |a|, |x^2 - 4av| \geq 1, |4av(x^2 - 4av)| < X/4, \\\notag &\hspace{7cm}\AND |x| < 2(X/4)^{1/4}(\log (X/4))^{1/4}\}.
\end{align}
Note that we have included the condition $y \equiv 0 \pmod{4}$ in the definition of $\R_a(X)$, and
\begin{equation}\label{bound on a} a^2 \leq |av| \ll x^2 + X^{1/2} \ll X^{1/2}(\log X)^{1/2}\mbox{ for all }(x,v)\in \R_a(X), \end{equation}
where the first $\ll$ follows from  Lemma~\ref{bound on y}. This means that
\[ \R_a(X) \mbox{ is the empty set when }|a| \gg X^{1/4}(\log X)^{1/4}. \]
From the discussion above, we now deduce that
\[ \sum_{r_2=0,1,2}N_1^{(r_2)}(X) = \sum_{1 \leq |a| \ll X^{1/4}(\log X)^{1/4}} \#(\R_a(X)\cap\bZ^2) + O\left(X^{3/4} (\log X)^{3/4}\right), \]
where the error term comes from Lemma \ref{tail cut}. The region $\R_a(X)$ is 2-dimensional and we may count the integral points inside via Davenport's lemma. To apply Davenport's lemma, we need to show that the projections of $\R_a(X)$ onto the coordinate axes are not too large:

\begin{lemma} \label{proj control} Let $a\in \bZ$ be non-zero. For any $(x,v) \in \R_a(X) \cap \bZ^2$, we have the bounds
\[|x| = O \left(X^{1/4} (\log X)^{1/4} \right)\AND |v| = O \left(X^{1/2} (\log X)^{1/2} |a|^{-1}\right),\]
where the implied constants are absolute. 
\end{lemma}
\begin{proof}
 The bound for $x$ is part of the definition of $\R_a(X)$, and that for $v$ follows from (\ref{bound on a}).
\end{proof}

From Lemma \ref{proj control} and Proposition \ref{Davenport}, we then obtain
\[\#(\R_a(X)\cap\bZ^2) = \Area(\R_a(X)) + O\left(X^{1/2} (\log X)^{1/2} |a|^{-1} + X^{1/4} (\log X)^{1/4}\right).\]
Let us now refine the above arguments to take into account the different real splitting types of integral binary quartic forms. First we show that:

\begin{lemma} \label{Wi real sig}For any $F\in V_{\J^{(1)}}(\bZ)$ as in (\ref{F in J}) of non-zero discriminant, the polynomial $F(x,1)$ has exactly $4-2r_2(F)$ real roots, where
\[ r_2(F) = \begin{cases} 0 & \mbox{if $AC, B^2-4AC,-AB>0$},\\ 1 &\mbox{if $AC<0$},\\ 2 &\mbox{otherwise}.\end{cases}\]
\end{lemma}
\begin{proof} For any integral binary quartic form $F$ given as in (\ref{generic F}) with non-zero discriminant, put
\[H(F) = 8a_4 a_2 - 3a_3^2\AND
S(F) = 3a_3^4 - 16a_4 a_3^2 a_2 + 16 a_4^2 a_2^2 + 16a_4^2 a_3 a_1 - 64a_4^3 a_0.\]
Then, by \cite[Proposition 7]{Cre} we know that $F(x,1)$ has exactly $4-2r_2(F)$ real roots, where
\[ r_2(F) = \begin{cases}0& \mbox{ if $\Delta(F),S(F),-H(F)>0$},\\ 1&\mbox{ if $\Delta(F)<0$},\\ 2 & \mbox{ otherwise}.\end{cases}\]
For $F\in V_{\J^{(1)}}(\bZ)$ as in (\ref{F in J}), we compute that
\[(H(F), S(F)) = (8AB,16A^2(B^2-4AC)),\]
and the claim now follows from (\ref{Ci}).
\end{proof}

For each $a\in\bZ$, in view of Lemma \ref{Wi real sig}, define
\begin{align*}
\R_a^{(0)}(X) & = \{(x,v)\in \R_a(X) : av,x^2 -4av,-ax>0\},\\
\R_a^{(1)}(X) & = \{(x,v)\in \R_a(X) : av < 0\},\\
\R_a^{(2)}(X) & = \R_a(X)\setminus\R_a^{(0)}(X)\setminus\R_a^{(1)}(X).
\end{align*}
Then, by the same arguments as above, we obtain
\[\label{N1 count}
    N_1^{(r_2)}(X) = \sum_{1 \leq |a| \ll X^{1/4}(\log X)^{1/4}} \#(\R_a^{(r_2)}(X) \cap \bZ^2) + O \left(X^{3/4} (\log X)^{3/4}\right),\]
and moreover
\[\#(\R_a^{(r_2)}(X) \cap \bZ^2) = \Area(\R_a^{(r_2)}(X)) + O \left(X^{1/2} (\log X)^{1/2} a^{-1} + X^{1/4} (\log X)^{1/4} \right).\]
It remains to compute the areas of these regions $\R_a^{(r_2)}(X)$ and then sum over $a$. We shall defer the area calculations to the next subsection.

\begin{proof}[Proof of Proposition \ref{N1X count}] It suffices prove the case $i=1$, and we have already shown that
\begin{align*} N_1^{(r_2)}(X) & = \sum_{1\leq |a|\ll X^{1/4}(\log X)^{1/4}}\Area(\R_a^{(r_2)}(X)) + O\left(X^{3/4}(\log X)^{3/4}\right)\\
&\hspace{1.5cm}+\sum_{1\leq |a|\ll X^{1/4}(\log X)^{1/4}}O\left(X^{1/2}(\log X)^{1/2}a^{-1}+X^{1/4}(\log X)^{1/4}\right).\end{align*}
The second summation is of order $O\left(X^{1/2}(\log X)^{3/2}\right)$ and hence is an error term. For the first summation, we use the areas of $\R_a^{(r_2)}(X)$ to be computed in the next subsection. In particular, for large values of $|a|$, by Proposition \ref{Area of Ra}, we have
\[ \sum_{X^{1/4}/2\leq |a|\ll X^{1/4}(\log X)^{1/4}} \Area(\R_a(X)) =  \sum_{X^{1/4}/2\leq |a|\ll X^{1/4}(\log X)^{1/4}}O\left(\frac{X}{a^2}\right) = O\left(X^{3/4}\right),\]
which is also an error term. For small values of $|a|$, by Corollary \ref{area cor}, we have
\begin{align*}&\hspace{5mm}\sum_{1\leq |a|\leq X^{1/4}/2}\Area(\R_a^{(r_2)}(X)) \\
& = 2\sum_{1\leq a\leq X^{1/4}/2}\left( \frac{1}{a}\frac{\mathfrak{r}(r_2)\sqrt{2}\Gamma(1/4)^2}{48\sqrt{\pi}}X^{3/4} + O\left(\frac{X^{3/4}(\log X)^{-1/4}}{a} + X^{1/2}\right)\right)\\
&= \frac{1}{4}\frac{\fr(r_2)\sqrt{2}\Gamma(1/4)^2}{24}X^{3/4}\log X + O\left(X^{3/4}(\log X)^{3/4}\right).
\end{align*}
This then proves the proposition.
\end{proof}

To conclude, we show that the contribution from forms in $V_{\J^{(3)}}(\bZ)(X)$ is negligible compared to $V_{\J^{(1)}}(\bZ)(X)$ and $V_{\J^{(2)}}(\bZ)(X)$.

\begin{proposition} \label{N3X} We have $\# V_{\J^{(3)}}(\bZ)(X) = O \left(X^{3/4} \right)$.
\end{proposition} 

\begin{proof} The proof proceeds similar to Proposition \ref{N1X}, and it suffices to estimate the number of points $(x,y)\in R(X)\cap \bZ^2$, but now we weight them by the function
\begin{equation}\label{kappa} d_\square(y) = \#\{(u,v)\in \bZ^2 : u^2 + v^2 = y\}\end{equation}
evaluated at $y$ rather than the divisor function. More specifically, for $|x| <\sqrt{2}X^{1/4}$ , by Lemma \ref{bound on y} (b) the constraint on $y$ is that it lies in the interval
\[ \left(\frac{x^2 - \sqrt{x^4 + 4X}}{2}, \frac{x^2 + \sqrt{x^4 + 4X}}{2} \right),\]
which has length $\sqrt{x^4 + 4X}$. Instead of (\ref{dr growth}), we apply the estimate
\[ \label{dr growth 2} \sum_{Y < y < Y + Y^\theta} d_\square(y) \sim \pi Y^\theta, \text{ valid for } \theta > 1/3,\]
by the work of Voronoi on the Gauss circle problem. We then see that the total contribution for each such $x$ is $O\left(X^{1/2} \right)$. Summing over $|x| <\sqrt{2} X^{1/4}$ gives us $O\left(X^{3/4}\right)$. For $|x| \geq \sqrt{2}X^{1/4}$, we use Lemma \ref{bound on y} (a), where now the total length of the intervals is $O \left(X/x^2\right)$. We argue as in the proof of Lemma \ref{tail cut}, but with the assumption $|x| \gg X^{1/4} (\log X)^{3/4}$ replaced by $|x| \gg X^{1/4}$, and the estimate (\ref{mid range cont}) becomes 
\[ X \sum_{X^{1/4} \ll |x| < X^{5/16}} x^{-2} = O \left(X^{3/4} \right)\]
instead. Finally, the estimate (\ref{big range cont}) in the proof of Lemma \ref{tail cut} readily shows that the contribution for $|x| \gg X^{5/16}$ is negligible. \end{proof}

\subsection{Area calculations}
\label{area comp}

In this subsection, we shall complete the proof of Proposition \ref{N1X count} by computing the areas of $\R_a^{(r_2)}(X)$. To do so, we shall first compute the areas of
\begin{align*} \label{Rr_2X} 
\R^{(0)}(X) & = \{(x,y) \in \R(X) : y>0\AND x^2-y>0\},\\
\R^{(1)}(X) & = \{(x,y) \in \R(X) : y <0\}, \notag \\
\R^{(2)}(X) & = \R(X)\setminus\R^{(0)}(X)\setminus\R^{(1)}(X). 
\end{align*}
It turns out that they are related to various \emph{elliptic integrals}. We shall therefore require some well-known explicit expressions for said integrals, given below. These are found in the book \cite{BF}. 

\begin{lemma}\label{preliminary int}We have
\begin{align*} 
\int_{0}^{1}\sqrt{z^4+1}dz &= \frac{1}{3}\left(\sqrt{2}+\frac{\Gamma(1/4)^2}{4\sqrt{\pi}}\right),\\
 \int_1^{\sqrt{2}(\log X)^{1/4}}\left(\sqrt{z^4+1}-z^2\right)dz&=\frac{1}{3}\left(\frac{-1}{\sqrt{2}+1}+\frac{\Gamma(1/4)^2}{4\sqrt{\pi}}\right) + O\left((\log X)^{-1/4}\right),\\
  \int_1^{\sqrt{2}(\log X)^{1/4}}\left(\sqrt{z^4+1}-\sqrt{z^4-1}\right)dz & = \frac{1}{3}\left(-\sqrt{2}+\frac{(1+\sqrt{2})\Gamma(1/4)^2}{4\sqrt{\pi}}\right)+O\left((\log X)^{-1/4}\right).
\end{align*}
\end{lemma}

\begin{proof}For any $z_1,z_2\in\bR$, using integration by parts, we have
\[\int_{z_1}^{z_2}\sqrt{z^4\pm1} dz= \frac{1}{3}\left(  \left[z\sqrt{z^4\pm1}\right]_{z_1}^{z_2} \pm 2\int_{z_1}^{z_2} \frac{dz}{\sqrt{z^4\pm1}}\right),\]
provided that the integrals are defined. We also have
\[\label{E int} \int_1^\infty \frac{dz}{\sqrt{z^4 +1}} = \frac{\Gamma(1/4)^2}{8 \sqrt{\pi}}\AND
 \int_1^\infty \frac{dz}{\sqrt{z^4 -1}}  = \frac{\sqrt{2}\Gamma(1/4)^2}{8\sqrt{\pi}},\]
which may be computed using complete elliptic integrals of the first kind; see \cite{BF}. The claims now follow from a straightforward computation.
\end{proof}

\begin{proposition}\label{Area of R} We have
\begin{align*}\Area(\R(X)) &= \frac{(2+2\sqrt{2})\Gamma(1/4)^2}{3\sqrt{\pi}}X^{3/4} + O\left(X^{3/4}(\log X)^{-1/4}\right),\\
\Area(\R^{(r_2)}(X))& = \fs(r_2)\cdot\frac{\sqrt{2} \Gamma(1/4)^2}{3 \sqrt{\pi}} X^{3/4} + O\left(X^{3/4} (\log X)^{-1/4}\right),\end{align*}
where $\fs(0) = \sqrt{2}$ and $\fs(1) = \fs(2) = 1$.
\end{proposition}
\begin{proof}Put $T(X) = (\log X)^{1/4}$. Since $\R^{(r_2)}(X)$ is symmetric along $x=0$ and
\[ \Area\left(\{(x,y)\in\bR^2 : \left(|y|\leq 1 \OR |x^2-y|\leq1\right) \AND |x|\leq 2X^{1/4}T(X)\}\right) = O\left(X^{1/4}T(X)\right),\]
we deduce from Lemma~\ref{bound on y} that
\begin{align*} \frac{1}{2}\Area(\R(X)) &= \int_{0}^{\sqrt{2}X^{1/4}}\int_{\frac{x^2-\sqrt{x^4+4X}}{2}}^{\frac{x^2+\sqrt{x^4+4X}}{2}}dydx
+\int_{\sqrt{2}X^{1/4}}^{2X^{1/4}T(X)} \int_{\frac{x^2-\sqrt{x^4+4X}}{2}}^{\frac{x^2-\sqrt{x^4-4X}}{2}}dydx\\
&\hspace{2.5cm}+\int_{\sqrt{2}X^{1/4}}^{2X^{1/4}T(X)} \int_{\frac{x^2+\sqrt{x^4-4X}}{2}}^{\frac{x^2+\sqrt{x^4+4X}}{2}}dydx + O\left(X^{1/4}T(X)\right),\\
\frac{1}{2}\Area(\R^{(1)}(X)) & = \int_{0}^{2X^{1/4}T(X)}\int_{\frac{x^2-\sqrt{x^4+4X}}{2}}^{0} dydx +O\left(X^{1/4}T(X)\right),\\
\frac{1}{2}\Area(\R^{(2)}(X)) &= \int_{0}^{2X^{1/4}T(X)}\int_{x^2}^{\frac{x^2+\sqrt{x^4+4X}}{2}}dydx+O\left(X^{1/4}T(X)\right).\end{align*}
By making a change of variables $x = \sqrt{2}X^{1/4}z$ and then applying Lemma~\ref{preliminary int}, we see that they simplify to the desired expressions, and the case $r_2=0$ follows as well.
\end{proof}

\begin{proposition}\label{Area of Ra}Let $a\in\bZ$ be non-zero. We have
\[ \Area(\R_a(X)) = O\left(X/a^2\right)\]
for $|a|\geq (X/4)^{1/4}/2$, and
\begin{align*}
\Area(\R_a^{(0)}(X)) &= \frac{1}{8|a|}\Area(\R^{(0)}(X/4)) + O\left(X^{1/2}\right)\\
\Area(\R_a^{(1)}(X)) &= \frac{1}{4|a|}\Area(\R^{(1)}(X/4))  + O\left(X^{1/2}\right)\\
\Area(\R_a^{(2)}(X)) &=\frac{1}{8|a|}\Area(\R^{(0)}(X/4))+\frac{1}{4|a|}\Area(\R^{(2)}(X/4)) + O\left(X^{1/2}\right)
\end{align*}
for $|a| < (X/4)^{1/4}/2$.
\end{proposition}
\begin{proof}Put $T(X) = (\log X)^{1/4}$. Observe that we have a map
\[\Phi_a:\R_a(X) \longrightarrow \R(X/4);\hspace{1em}\Phi_a(x,v) = (x,4av),\]
whose Jacobian matrix has determinant $4a$, and its image is given by
\begin{align*}\A_a(X/4) &= \{(x,y)\in \bR^2 : |y|> 4a^2, \, |x^2-y|\geq1,\, |y(x^2-y)|< X/4,\,\\&\hspace{6.5cm}\AND|x|\leq 2(X/4)^{1/4}T(X/4)\}.\end{align*}
Suppose that $|a|\geq (X/4)^{1/4}/2$. For all $(x,y)\in \A_a(X/4)$, we then have $|y|\geq (X/4)^{1/2}$, and 
\[|x|\in \left(\sqrt{y - \frac{X}{4y}},  \sqrt{y + \frac{X}{4y}}\right),\]
an interval of length $O\left(X/y^{3/2}\right)$. This implies that
\[ \Area(\A_a(X/4)) = O\left(\int_{4a^2}^{\infty}\frac{X}{y^{3/2}}dy\right) = O\left(X/a\right),\]
and the claim follows because $\Phi_a$ has determinant $4a$. Suppose next that 
$|a| < (X/4)^{1/4}/2$. For all $(x,y)\in \R(X/4) \setminus \A_a(X/4)$, we then have $|y| < (X/4)^{1/2}$, and
\[ |x|\in \left[ 0,\sqrt{y+\frac{X}{4|y|}}\right),\]
an interval of length $O(X^{1/2}/|y|^{1/2})$. This implies that
\[\Area(\R(X/4)\setminus\A_a(X/4)) = O\left(\int_{-4a^2}^{4a^2}\frac{X^{1/2}}{|y|^{1/2}} \right) = O(aX^{1/2}).\]
Now, by definition, it is clear that
\begin{align*}
\Phi_a(\R_a^{(0)}(X)) & = \A_a(X/4) \cap \R^{(0)}(X/4)\cap \{(x,y)\in\bR^2: ax<0\}, \\
\Phi_a(\R_a^{(1)}(X)) & = \A_a(X/4)\cap \R^{(1)}(X/4).
\end{align*}
Note that $\R^{(0)}(X)$ is symmetric along $x=0$ and the condition $ax<0$ only imposes that $x$ has opposite sign as $a$. We hence have
\begin{align*}
    \Area(\Phi_a(\R_a^{(0)}(X))) &= \frac{1}{2}\Area(\R^{(0)}(X/4)) + O\left(aX^{1/2}\right),\\
    \Area(\Phi_a(\R_a^{(1)}(X))) & = \Area(\R^{(1)}(X/4)) + O\left(aX^{1/2}\right).
\end{align*}
The claim then follows $\Phi_a$ has determinant $4a$, and the case $r_2=2$ then follows as well. 
\end{proof}

Combining Propositions \ref{Area of R} and \ref{Area of Ra}, we obtain:

\begin{corollary}\label{area cor}Let $a\in\bZ$ be non-zero with $|a| < X^{1/4}/2$. Then
\[    \Area(\R_a^{(r_2)}(X)) = \frac{1}{|a|} \frac{\mathfrak{r}(r_2)\sqrt{2}\Gamma(1/4)^2}{48\sqrt{\pi}}X^{3/4} + O\left(\frac{X^{3/4}(\log X)^{-1/4}}{|a|} + X^{1/2}\right),\]
where $\mathfrak{r}(0) =1$, $\mathfrak{r}(1) = \sqrt{2}$, and $\mathfrak{r}(2) = 1 + \sqrt{2}$ are as in Theorem \ref{real split MT}.
\end{corollary}

\section{Criteria for pre-maximal forms to be maximal}
\label{max section} 

In the previous section, we counted elements in the three families $V_{\J^{(i)}}(\bZ)$ for $i=1,2,3$, the so-called pre-maixmal forms. To obtain Theorem~\ref{real split MT}, we shall require more refined arithmetic considerations, and understand when a form $F$ is maximal, namely when its associated quartic ring $Q_F$ is a maximal. \\

Given any $F\in V_{\J^{(i)}}(\bZ)$, the quartic ring $Q_F$ is maximal if and only if $Q_F$ is \emph{maximal at $p$}, namely $\bZ_p\otimes_\bZ Q_F$ is maximal over $\bZ_p$, for every prime $p$. We have the following explicit criteria:

\begin{theorem}\label{J1 thm}Let $F\in V_{\J^{(1)}}(\bZ)$ be as in (\ref{F in J}).
\begin{enumerate}[(a)]
\item The quartic ring $Q_F$ is maximal at an odd prime $p$ if and only if $p$ does not divide all of $A,B,$ and $C$, as well as  $A,C,B^2-4AC\not\equiv0 \pmod{p^2}$.
\item The quartic ring $Q_F$ is maximal at $2$ if and only if $A,B,$ and $C$ are not all even, as well as $A,C\not\equiv0\pmod{4}$ and
\[\begin{cases}A+B+C\not\equiv0\hspace{-3mm}\pmod{4}&\mbox{when one of $A,B,C$ is even},\\
A+B\equiv0\hspace{-3mm}\pmod{4}\OR B+C\equiv0\hspace{-3mm}\pmod{4}&\mbox{when $A,B,C$ are all odd}.
\end{cases}\]
\end{enumerate}
\end{theorem}

\begin{theorem}\label{J2 thm}Let $F\in V_{\J^{(2)}}(\bZ)$ be as in (\ref{F in J}).
\begin{enumerate}[(a)]
\item The quartic ring $Q_F$ is maximal at an odd prime $p$ if and only if $p$ does not divide all of $A,B,$ and $C$, as well as  $4A\pm2B+C,B^2-4AC\not\equiv0 \pmod{p^2}$.
\item The quartic ring $Q_F$ is maximal at $2$ if and only if $A,B,$ and $C$ are not all even, as well as $2B+C\not\equiv0\pmod{4}$, and $A,A-B+C\not\equiv0\pmod{4}$ when $C$ is odd and $B$ is even.
\end{enumerate}
\end{theorem}

\begin{theorem}\label{J3 thm}Let $F\in V_{\J^{(3)}}(\bZ)$ be as in (\ref{F in J}).
\begin{enumerate}[(a)]
\item The quartic ring $Q_F$ is maximal at an odd prime $p$ if and only if $p$ does not divide all of $A,B,$ and $C$, as well as $B^2-4AC\not\equiv0 \pmod{p^2}$ and
\[\begin{cases} 
(4A-C)^2 + 4B^2\not\equiv0\hspace{-3mm}\pmod{p^2} &\mbox{when $p$ divides both $4A-C$ and $B$},\\
(4A-C)^2 + 4B^2\not\equiv0\hspace{-3mm}\pmod{p^2} &\mbox{when $p$ does not divide both $4A-C $ and $B$}.
\end{cases}\]
\item The quartic ring $Q_F$ is maximal at $2$ if and only if $A,B,$ and $C$ are not all even, as well as $C\not\equiv0\pmod{4}$, and $A,A-B+C\not\equiv0\pmod{4}$ when $C$ is odd and $B$ is even.
\end{enumerate}
\end{theorem}

Recall that the number of forms in $V_{\J^{(3)}}(\bZ)$ is negligible by Proposition~\ref{N3X}. Hence, to prove Theorem~\ref{real split MT} we actually do not need Theorem~\ref{J3 thm}, and we included it here just for completeness.\\

To prove the above theorems, we shall use a result of Bhargava \cite{HCL3}. For each \emph{nondegenerate} pair $(U,V)$ of integral ternary quadratic forms and each prime $p$, let $((U,V),p)$ denote the symbol as in \cite[(35)]{HCL3}; see \cite[Section 4.2]{HCL3} for the definition. The next proposition is proven in \cite[Section 4.2]{HCL3}; there are other possibilities for $((U,V),p)$, but the five listed below are the only ones which we need.

\begin{proposition} \label{maximality types} Let $(U,V)$ be a pair of integral ternary quadratic forms and let $Q$ denote its associated quartic ring. Let $p$ be a prime.
\begin{enumerate}[(a)]
\item If $(U,V)$ is degenerate at $p$, then $Q$ is not maximal at $p$.
\item If $((U,V),p) =(1^211)$ or $(1^22)$, then $(U,V)$ is $\GL_2(\bZ)\times\GL_3(\bZ)$-equivalent to a pair
\[ \left(\begin{pmatrix} u_{11} &* & * \\ * & * & * \\ * & * & *\end{pmatrix},\begin{pmatrix} v_{11} & * & * \\ \frac{v_{12}}{2} & * & * \\ \frac{v_{13}}{2} & * & *\end{pmatrix} \right)\mbox{ with }u_{11},v_{11},v_{12},v_{13}\equiv0\hspace{-3mm}\pmod{p},\]
and $Q$ is maximal at $p$ precisely when $v_{11}\not\equiv0\pmod{p^2}$.
\item If $((U,V),p) = (2^2), (1^4),$ or $(1^21^2)$, then $(U,V)$ is $\GL_2(\bZ)\times\GL_3(\bZ)$-equivalent to a pair
\[\left(\begin{pmatrix} u_{11}& * & * \\ \frac{u_{12}}{2} & u_{22} & * \\  *& * & *\end{pmatrix},\begin{pmatrix} v_{11} & * & * \\  \frac{v_{12}}{2}& v_{22} & * \\ \frac{v_{13}}{2} & \frac{v_{23}}{2}& *\end{pmatrix} \right)\mbox{ with } v_{11},v_{22},v_{12},v_{13},v_{23}\equiv0\hspace{-3mm}\pmod{p}.\]
Denote the pair above by $(U',V')$ and write $\S$ for the set of intersection points of $U'=0$ and $V'=0$ in the projective $2$-space over the finite field $\bF_p$ of $p$ elements. Given $P\in \S$, let $S_P\in\GL_3(\bZ)$ be any matrix which sends $(1,0,0)$ to $P$, and write
\[  (U',V')_{S_P}= \left(\begin{pmatrix} * &* & * \\ * & * & * \\ * & * & *\end{pmatrix},\begin{pmatrix} v_{P,11} &*& * \\ * & * & * \\ * & * & *\end{pmatrix} \right).\]
\end{enumerate}
\begin{enumerate}[(i)] 
\item If $((U,V),p) = (2^2)$, then $\S$ is empty, and $Q$ is maximal at $p$ precisely when the two vectors $(v_{11}/p, v_{12}/p, v_{22}/p)$ and $(u_{11}, u_{12}, u_{22})$ are linearly independent over $\bF_p$.
\item If $((U,V),p)=(1^4)$ or $(1^21^2)$, then $\S$ has cardinality one or two, respectively, and $Q$ is maximal precisely when $v_{P,11}\not\equiv0\pmod{p^2}$ for all $P\in\S$.
\end{enumerate}
\end{proposition}

\begin{proof}[Proof of Theorem \ref{J1 thm}, \ref{J2 thm}, and \ref{J3 thm}]Let $F\in V_{\J^{(i)}}(\bZ)$ be as in (\ref{F in J}). Note that:
\begin{itemize}
\item If $p$ divides all of $A,B,C$, then $Q_F$ is not maximal at $p$ by Proposition~\ref{maximality types} (a).
\item If $p$ does not divide $\Delta(F)$, then $Q_F$ is maximal at $p$.
\end{itemize}
For other primes $p$, the claims may be verified using Proposition~\ref{maximality types} by a straightforward computation. Let us just remark that for odd primes $p$, the case $i=2$ follows from the case $i=1$ in view of (\ref{map 2 to 1}) and (\ref{iota explicit}). 
\end{proof}

Theorems~\ref{J1 thm},~\ref{J2 thm}, and~\ref{J3 thm} have the following immediate consequence:

\begin{corollary}\label{fund disc}For $i=1,2,3$, let $F\in V_{\J^{(i)}}(\bZ)$ be as in (\ref{F in J}) such that $Q_F$ is maximal. Then, necessarily $B^2-4AC$  is a fundamental discriminant.
\end{corollary}

Together with (\ref{Ci}), they also imply the following corollary:

\begin{corollary}\label{V4 forms} For $i=1,2,3$, let $F\in V_{\J^{(i)}}(\bZ)$ be such that $\Delta(F)$ is a square and that $Q_F$ is maximal. Then, necessarily $F(x,y) = Ax^4 + Bx^2y^2 + Ay^4$.
\end{corollary}

We conclude this section by computing the local densities of the forms $F\in V_{\J^{(i)}}(\bZ)$ having maximal $Q_F$. We shall ignore the case $i=3$ because of Proposition \ref{N3X}. \\

For each $(A,B,C)\in\bZ^3$, for convenience let us write
\begin{align}\label{(A,B,C)}
 (A,B,C)_1 & = Ax^4 + Bx^2y^2 + Cy^4\\\notag
 (A,B,C)_2 & = Ax^4 + Bx^3y + (C+2A)x^2y^2 + Bxy^3 + Ay^4
\end{align}
for the corresponding form in $V_{\J^{(1)}}(\bZ)$ and $V_{\J^{(2)}}(\bZ)$, respectively. In view of Theorems \ref{J1 thm} and \ref{J2 thm}, for both $i=1,2$, whether $(A,B,C)_i$ is maximal at $p$ depends only the values of $A,B,C$ mod $p^2$. Let us consider the linear forms
\[ \ell_1(A,B,C) = A \AND \ell_2(A,B,C) = 4A-2B+C.\]
For each $a\in\bZ$ and each prime $p$, define
\begin{align*}
     \rho_a^{(i)}(p)& = \#\{(A,B,C)\in(\bZ/p^2\bZ)^3: \ell_i(A,B,C) \equiv a\ppmod{p^2}\\
     &\hspace{4.5cm}\AND (A,B,C)_i\mbox{ is not maximal at $p$}\}.
\end{align*}
Recall that the isomorphism $\iota$ in (\ref{map 2 to 1}) and (\ref{iota explicit}) has determinant $64$. Hence, for all odd primes $p$, it induces an isomorphism from $V_{\J^{(2)}}(\bZ/p^2 \bZ)$ to $V_{\J^{(1)}}(\bZ/p^2 \bZ)$ which preserves maximality of the forms at $p$. In particular, we have the equality
\[ \rho_a^{(1)}(p) = \rho_a^{(2)}(p) \mbox{ for all odd primes $p$}.\]
Notice that $(A,B,C)_i$ with $\ell_i(A,B,C) = a$ is never maximal when $a$ is not square-free. Hence, it suffices to consider the square-free values of $a$, and a direct computation using Theorems \ref{J1 thm} and \ref{J2 thm} yield:

\begin{proposition} \label{rhoa eva} Let $a\in\bZ$ be square-free. We have
\[ \rho_{a}^{(1)}(p) = \rho_{a}^{(2)}(p) = p(2p-1)\mbox{ for all odd primes $p$, and  } \rho_{a}^{(1)}(2) = 8,\, \rho_{a}^{(2)}(2) = 4.\]
Thus, the density of pairs $(b,c)$ in $\bZ^2$ such that $(a,b,c)_i$ is maximal at a prime $p$ is $1 - (2p-1)p^{-3}$ for $p$ odd and $i=1,2$, $1/2$ for  $p = 2$ and $i = 1$, and $3/4$ for $p = 2$ and $i = 2$.
\end{proposition}

Finally, in view of Corollary \ref{V4 forms}, for each prime $p$, let us also define
\[ \rho_0(p) = \#\{(A,B)\in (\bZ/p^2\bZ)^2: Ax^4 + Bx^2y^2 + Ay^4\mbox{ is not maximal at $p$}\}.\]
Using Theorem \ref{J1 thm}, it is easy to verify that:

\begin{proposition} \label{rho eva} We have $\rho_{0}(p) = p(4p - 3)$ for all odd primes $p$, and $\rho_{0}(2) = 9$.
Thus, the density of pairs $(a,b)$ in $\bZ^2$ such that $ax^4 + bx^2 y^2 + ay^4$ is maximal at a prime $p$ is $1 - (4p-3)p^{-3}$ for $p$ odd, and $7/16$ for $p=2$. 
\end{proposition}

\section{Counting maximal dihedral binary quartic forms by conductor}
\label{MT proof}

In this section, we shall prove our main result Theorem~\ref{real split MT}. Recall from Proposition \ref{N3X} that the contribution from $V_{\J^{(3)}}(\bZ)(X)$ is negligible. For $i=1,2$, we shall modify the arguments in Subsection \ref{pre-count section} to enumerate only the maximal forms 
\[ V_{\J^{(i)}}^{\mmax}(\bZ)(X) = \{F\in V_{\J^{(i)}}^{\mmax}(\bZ) : 0 < |\C_{\J^{(i)}}(F)| < X\},\]
and this amounts to applying a square-free sieve by Theorems \ref{J1 thm} and \ref{J2 thm}. We shall also prove that the forms which have abelian Galois group or are reducible contribute to an error term.


\subsection{A square-free sieve} 
\label{sieve}

In Subsection \ref{pre-count section}, we saw that the enumerations of $V_{\J^{(1)}}(\bZ)(X)$ and $V_{\J^{(2)}}(\bZ)(X)$ may be reduced to counting integral points in the set $\R_a(X)$ defined in (\ref{Ra definition}), and then summing over $1\leq |a|\ll X^{1/4}(\log X)^{1/4}$. To extract only the points which correspond to maximal forms, by Theorems \ref{J1 thm} and \ref{J2 thm}, we essentially have to apply a square-free sieve. \\

Recall the notation in (\ref{(A,B,C)}) and let us focus on $i=1$. For each non-zero $a\in\bZ$, we wish to determine the number
\[ M_{a}(X) = \#\{(b,c)\in \R_a(X)\cap \bZ^2 : (a,b,c)_1 \mbox{ is maximal}\}. \]
We introduce the following two quantities:
\begin{align*}
M_{a,1}(X) & = \#\{(b,c)\in\R_a(X)\cap\bZ^2: (a,b,c)_1 \mbox{ is not maximal at $p$}\ \Rightarrow p > \xi(X)\},\\
M_{a,2}(X) & = \# \{(b,c) \in \R_a(X) \cap \bZ^2 : (a,b,c)_1\text{ is not maximal at } p \Rightarrow p > \xi(X)\\
&\hspace{3cm}\AND \exists p > \xi(X) \text{ such that } (a,b,c)_1 \text{ is not maximal at } p \},
\end{align*}
where $\xi(X)$ shall be suitably chosen later. We note now that for all $\delta > 0$ we can choose $\xi(X)$ so that if $m$ is square-free and divisible only by primes at most $\xi(X)$, then $m \ll_\delta X^{\delta}$. Indeed, we can choose $\xi(X) = (\delta/2) \log X$ by the well-known fact that
\[\prod_{p \leq z} p \leq e^{2 z}\]
for any positive number $z$. \\

We have the fundamental sieve inequalities
\begin{equation}\label{sieve inequality} M_{a,1}(X) - M_{a,2}(X) \leq M_a(X) \leq M_{a,1}(X).\end{equation}
It then suffices to show that $M_{a,2}(X)$ is small compared to $M_{a,1}(X)$. Note that we only need to consider square-free values of $a$, for otherwise $M_{a}(X)$ is zero by Theorem \ref{J1 thm}. 

\begin{lemma}\label{bound on p}Let $a\in\bZ$ be non-zero and square-free. For any $(b,c)\in\R_a(X)\cap \bZ^2$, the form $(a,b,c)_1$ is not maximal at a prime $p$ only if $p \ll X^{1/4}(\log X)^{1/4}$.
\end{lemma}
\begin{proof}Suppose that $(a,b,c)_1$ is not maximal at a prime $p$. Since $a$ is square-free, by Theorem \ref{J1 thm} this means that $p^2$ divides either $c$ or $b^2-4ac$. But
\[ |b| = O\left(X^{1/4}(\log X)^{1/4}\right) \AND |c| = O\left(X^{1/2}(\log X)^{1/2}a^{-1}\right)\]
by Lemma \ref{proj control}, so we see that the stated bound for $p$ holds.
\end{proof}

Let $a\in\bZ$ be non-zero. For  $m\in\bN$, let us put
\[ M_a(m;X) = \#\{(b,c)\in \R_a(X)\cap \bZ^2: (a,b,c)_1 \mbox{ is not maximal at $p$ for all $p\mid m$}\}\]
if $m$ is square-free, and $M_a(m,X) = 0$ otherwise. We then have
\[M_a(X) = \sum_{\substack{m \in \bN }} \mu(m) M_a(m; X)\AND M_{a,1}(X) = \sum_{\substack{m \in \bN \\ p \mid m \Rightarrow p \leq \xi(X)}} \mu(m) M_a(m ; X),\]
where $\mu(\cdot)$ is the M\"{o}bius function. Let us also define the local density
\[ \rho_a(m) = \#\{(b,c)\in (\bZ/m^2\bZ)^2: (a,b,c)_1\mbox{ is not maximal at $p$ for all $p\mid m$}\}\]
if $m$ is square-free, and $\rho_a(m) = 0$ otherwise. Note that $\rho_a(\cdot)$ is multiplicative, and $\rho_a(p) = \rho_a^{(1)}(p)$ for all primes $p$, where $\rho_a^{(1)}(\cdot)$ is defined as in Section \ref{max section}. By Proposition \ref{rho eva}, for $a$ square-free, we then have $\rho_a(p) = p(2p-1)$ whenever $p$ is odd.

\begin{proposition} \label{Ma2 prop} For any non-zero and square-free $a\in\bZ$, we have
\[ M_{a,2}(X) = O \left(\frac{X^{3/4}}{a \xi(X)} + \frac{X^{1/2}(\log X)^{1/2}}{\log X} \right).\]
\end{proposition}

\begin{proof} Put $T(X) = (\log X)^{1/4}$. Lemma \ref{bound on p} implies that
\[\label{Ma2 est} M_{a,2}(X) \ll \sum_{\xi(X) < p \ll X^{1/4}T(X)} M_a(p; X).\]
For each prime $p$, there are $O(p^2)$ elements $(b,c) \in (\bZ/p^2 \bZ)^2$ such that $(a,b,c)_1$ is not maximal. Since $a$ is square-free by assumption, we know from Theorem \ref{J1 thm} that $(a,b,c)_1$ is not maximal at odd $p$ if and only if one of the following holds:
\[ (1)\hspace{2mm}p\mid \gcd(a,b,c),\hspace{5mm}(2)\hspace{2mm} p^2 \mid c,\hspace{5mm} (3)\hspace{2mm}p^2 \mid b^2 - 4ac.\]
We shall require the following trivial but important observation: for any pair of positive integers $k,m$ and an interval $I$ of length $Y$, we have: 
\begin{equation} \label{int cong} \# \{n \in I \cap \bZ : n\equiv k \ppmod{m} \} = \frac{Y}{m} + O(1).\end{equation}
\begin{enumerate}[(1)]
\item For $p\mid \gcd(a,b,c)$, we note that there are only $\omega(a)$ possibilities for $p$, where $\omega(\cdot)$ is the number of distinct prime divisors. For each such $p$ and $|b| \leq 2X^{1/4}$, there are then $O\left(b^2/(ap) + 1\right)$ possibilities for $c = pc^\prime$, say. Similarly, for each $b$ with $|b| \gg X^{1/4}$ there are $O\left(X/(apb^2) + 1 \right)$ possibilities for $c$. We are then led to the sum
\[ \label{p div a}\sum_{\substack{p \mid a \\ p > \xi(X)}} \left(\sum_{\substack{|b| \leq 2X^{1/4} \\ p \mid b}} O \left(\frac{b^2}{ap} + 1  \right) + \sum_{\substack{2X^{1/4} < |b| \ll X^{1/4} T(X) \\ p \mid b}} O \left(\frac{X}{apb^2} + 1 \right) \right).\]
We note that since $p \mid a$ and $|a| \ll X^{1/4} T(X)$ by assumption, we obtain the estimate
\[O \left(\frac{X^{3/4}}{a \xi(X)} + X^{1/4} T(X)  \right),\]
which is sufficient for our purposes. 
\item For $p^2 \mid c$, we proceed as before and note that for each $|b| \leq 2X^{1/4}$ there are $O\left(b^2/(ap^2) + 1\right)$ possibilities for $c$, and for each $|b| > 2X^{1/4}$ there are $O \left(X/(ap^2 b^2) + 1 \right)$ possibilities. We note that since $ |ac| = O \left(X^{1/2} T(X)^2 \right)$ by Lemma \ref{proj control} and $p^2 \mid c$, it follows that $p \ll X^{1/4} T(X)/|a|$. Hence we are led to the sum
\[ \label{p2 div c} \sum_{\xi(X) < p \ll X^{1/4} T(X)/a} \left(\sum_{|b| \leq 2X^{1/4}} O \left(\frac{b^2}{ap^2} + 1  \right) + \sum_{2X^{1/4} < |b| \ll X^{1/4} T(X)} O \left(\frac{X}{ap^2 b^2} + 1 \right)  \right).\]
This then gives the estimate
\[O \left(\frac{X^{3/4}}{a \xi(X)} + \frac{X^{1/2}T(X)^2}{a \log X} \right).\]
\item For $p^2 \mid  b^2 - 4ac$, we may assume that $p\nmid a$, for otherwise $p\mid \gcd(a,b,c)$ because $a$ is square-free, and we are in the first case. We may further assume that $p\nmid c$, for otherwise $p^2 \mid c$ and we are in the previous case. Note that $b^2$ and $4ac$ are both bounded by $O (X^{1/2} T(X)^2)$ by Lemma \ref{proj control}, and since $p^2 \mid b^2 - 4ac$, we must have $p \ll X^{1/4} T(X)$. The argument is then very similar to the previous case, and we get the sum
\[\label{p2 div bac} \sum_{\xi(X) < p \ll X^{1/4} T(X)} \left(\sum_{|b| \leq 2X^{1/4}} O \left(\frac{b^2}{ap^2} + 1 \right) + \sum_{2X^{1/4} < |b| \ll X^{1/4} T(X)} O \left(\frac{X}{ap^2 b^2} + 1 \right) \right).\]
The corresponding estimate is 
\[O \left(\frac{X^{3/4}}{a \xi(X)} + \frac{X^{1/2} T(X)^2}{\log X} \right).\]
\end{enumerate}This completes the proof.
\end{proof}

Our next task is to estimate $M_{a,1}(X)$ and show that, with suitably chosen $\xi(X)$, the sum
\[\sum_{1 \leq |a| \ll X^{1/4} (\log X)^{1/4}} M_{a,1}(X)\]
will constitute the main term. From Lemma \ref{proj control} and (\ref{int cong}), we deduce the following estimate:
\[ \label{Mam est} M_a(m; X) = \frac{\rho_a(m) \#(\R_a(X) \cap \bZ^2)}{m^4} + O \left(\rho_a(m) \left(\frac{X^{1/4} (\log X)^{1/4}}{m^2} + \frac{X^{1/2} (\log X)^{1/2}}{am^2} + 1\right)\right). \]
By the multiplicativity of $\rho_a(\cdot)$ and the estimate $\rho_a(m) = O_\ep \left(m^{2+\ep} \right)$, it follows that
\begin{align*}
M_{a,1}(X)  & =  \#(\R_a(X) \cap \bZ^2) \prod_{p \leq \xi(X)} \left(1 - \frac{\rho_a(p)}{p^4} \right) \\
& \hspace{2cm} + \sum_{\substack{m \in \bN \\ p | m \Rightarrow p \leq \xi(X)}} O_\ep \left( \frac{X^{1/2} (\log X)^{1/2} m^\ep}{|a|} + m^{2 + \ep} \right).\end{align*}
By our earlier remark prior to (\ref{sieve inequality}), we note that the latter sum over $m$ can be majorized by the sum over $m \ll_\delta X^\delta$ for any $\delta > 0$, whence
\[ \sum_{\substack{m \in \bN \\ p | m \Rightarrow p \leq \xi(X)}} O_\ep \left( \frac{X^{1/2} (\log X)^{1/2} m^\ep}{|a|} + m^{2 + \ep} \right)  = O_{\delta, \ep} \left( \sum_{m \ll_\delta X^\delta}  \left(\frac{X^{1/2} (\log X)^{1/2} m^\ep}{|a|} + m^{2 + \ep} \right) \right) \]
and it follows that that
\[\sum_{|a| \ll X^{1/4}(\log X)^{1/4}} O_{\delta, \ep} \left( \frac{X^{1/2} (\log X)^{1/2} m^\ep}{|a|} + m^{2 + \ep} \right) =  O_{\delta,\ep}\left(X^{1/2+\delta+\ep}+X^{1/4+3\delta+\ep}\right). \]
We choose the optimal value of $\delta$ to get an optimal error term for the above expression; this requires $\delta = 1/8$, and correspondingly $\xi(X)=(\log X)/16$. We then deduce that
\[\label{Ma1 est fin} M_{a,1}(X) = \#(\R_a(X) \cap \bZ^2) \cdot \prod_{p \leq \xi(X)} \left(1 - \frac{\rho_a(p)}{p^4} \right) + O_\ep \left(X^{1/2 + \ep} |a|^{-1} + X^{3/8 + \ep} \right).\]
Putting the estimates for $M_{a,1}(X), M_{a,2}(X)$ and noting (\ref{sieve inequality}), we obtain
\begin{align*}M_a(X) & = \# (\R_a(X) \cap \bZ^2) \cdot  \prod_{p \leq \xi(X)} \left(1 - \frac{\rho_a(p)}{p^4} \right) + O_\ep \left(X^{5/8 + \ep} \right) \\
&\hspace{4.5cm}+ O \left(\frac{X^{3/4} \log X}{\xi(X)} +  X^{3/8} (\log X)^{3/8}\right). 
\end{align*}
Finally, we note that the infinite product $\prod_{p > \xi} (1 - \rho_a(p) p^{-4})$ satisfies 
\[\prod_{p > \xi} \left(1 - \frac{\rho_a(p)}{p^4} \right) = O \left( \exp \left(- \sum_{p > \xi} p^{-2} \right) \right) = O \left(1 + \xi^{-1} \right). \]
It thus follows that 
\[ \prod_p \left(1 - \frac{\rho_a(p)}{p^4} \right)  = \prod_{p \leq \xi(X)} \left(1 - \frac{\rho_a(p)}{p^4} \right) \left(1 + O \left(\xi(X)^{-1} \right) \right),
\]
and by substituting $\rho_a(p) = \rho_a^{(1)}(p)$, we then obtain 
\[\label{Max est pen} M_a(X) = \# (\R_a(X) \cap \bZ^2) \cdot  \prod_{p} \left(1 - \frac{\rho_a^{(1)}(p)}{p^4} \right) + O_\ep \left(X^{3/8 + \ep} + X^{3/4}a^{-1}\right), \]
where we have suppressed terms strictly smaller than the ones displayed.\\

The above arguments apply equally well if we take into account the real splitting types of the integral binary quartic forms, namely if we consider instead
\[ M_a^{(r_2)}(X) = \#\{(b,c) \in \R_a^{(r_2)}(X)\cap \bZ^2 : (a,b,c)_1\mbox{ is maximal}\}.\]
For each non-zero and square-free $a\in\bZ$, we have the more refined formula
\[M_a^{(r_2)}(X) = \# (\R_a^{(r_2)}X) \cap \bZ^2) \cdot  \prod_{p} \left(1 - \frac{\rho_a^{(1)}(p)}{p^4} \right) + O_\ep \left(X^{3/8 + \ep} + X^{3/4}a^{-1}\right).\]
In fact, using the same proof we may easily generalize this to count maximal forms in $V_{\J^{(2)}}(\bZ)$, that is if we consider the number
\[ M_a^{\star,(r_2)}(X) = \#\{(b,c) \in \R_a^{(r_2)}(X)\cap \bZ^2 : (a,b,c)\in\Lambda_\iota\AND\iota^{-1}(a,b,c)_2\mbox{ is maximal}\}.\]
Here $\iota$ is the map in (\ref{iota explicit}) and $\Lambda_\iota$ is the lattice in $\bZ^3$ defined after it. Recall that $\det(\Lambda_\iota) = 64$, and notice that the number of pairs $(b,c)\mod{16}$ with  $(a,b,c)\in\Lambda_\iota$ is independent of $a$. Thus, the condition that $(a,b,c)\in\Lambda_\iota$ contributes to an extra factor of $1/64$. For each non-zero and square-free $a\in\bZ$, a similar argument as above then shows that
\[M_a^{\star,(r_2)}(X) = \# (\R_a^{(r_2)}(X) \cap \bZ^2) \cdot\frac{1}{64}\cdot  \prod_{p} \left(1 - \frac{\rho_a^{(2)}(p)}{p^4} \right) + O_\ep \left(X^{3/8 + \ep} + X^{3/4}a^{-1}\right),\]
where $\rho_a^{(2)}(\cdot)$ is defined as Section \ref{max section}. Using the values of $\rho_a^{(1)}(p)$ and $\rho_a^{(2)}(p)$ given by Proposition \ref{rho eva}, we may rewrite the above formulae as
\begin{align}\label{precount}
    M_{a}^{(r_2)}(X) & = \#(\R_a^{(r_2)}(X)\cap\bZ^2)\cdot \frac{4}{5}\prod\left(1-\frac{2p-1}{p^3}\right) + O_\ep \left(X^{3/8 + \ep} + X^{3/4}a^{-1}\right),\\\notag
    M_{a}^{\star,(r_2)}(X) & = \#(\R_a^{(r_2)}(X)\cap\bZ^2)\cdot \frac{1}{64}\cdot\frac{6}{5}\prod\left(1-\frac{2p-1}{p^3}\right) +  O_\ep \left(X^{3/8 + \ep} + X^{3/4}a^{-1}\right).
\end{align}
We are now almost ready to prove Theorem \ref{real split MT}.

\subsection{From counting forms to counting equivalence classes} So far we have been counting integral binary quartic forms, but what we actually want are their $\GL_2(\bZ)$-equivalence classes. The next lemma deals with this issue:

\begin{lemma}\label{representative cor}For $i=1,2$, let $F\in V_{\J^{(i)}}(\bZ)$ be such that $\Delta(F)$ is not a square. Then the forms in $V_{\J^{(i)}}(\bZ)$ which are $\GL_2(\bZ)$-equivalent to $F$ are exactly $F_T$, where
\begin{equation}\label{T choices}
\pm T = \begin{pmatrix} 1&0 \\ 0&1 \end{pmatrix},\begin{pmatrix} 1 & 0 \\ 0 &-1 \end{pmatrix},\begin{pmatrix} 0 & 1 \\ 1 & 0 \end{pmatrix},\begin{pmatrix} 0 & 1\\ -1 &0 \end{pmatrix}.
\end{equation}
In particular, when $-\Delta(F)$ is also not a square, for $i=1,2$, respectively, the form $F$ is $\GL_2(\bZ)$-equivalent to a unique element in $V_{\J^{(i)}}(\bZ)$ of the shape
\[\begin{cases}ax^4 + bx^2y^2 + cy^4 \mbox{ with }|c| > |a|,\\
ax^4 + bx^3y + (c+2a)x^2y^2 + bxy^3 + ay^4\mbox{ with }|4a+2b+c| > |4a-2b+c|.\end{cases}\]
\end{lemma}
\begin{proof}Suppose that $T\in \GL_2(\bZ)$ and $F_T\in V_{\J^{(i)}}(\bZ)$. This means that $TM_{\J^{(i)}}T^{-1}$ also fixes $F$. Since $F\in V_{\J^{(i)}}(\bZ)$ and $\Delta(F)$ is not square by hypothesis, from \cite[Theorem 1.1]{TX} we deduce that $TM_{\J^{(i)}}T^{-1} = \pm M_{\J^{(i)}}$. For both $i=1,2$, a simple calculation shows that $\pm T$ must be one of the four stated matrices. By considering their action on binary quartic forms, we then see from (\ref{F in J}) that for $i=1,2$, respectively, either $F$ or $F_T$ must have the shape
\[\begin{cases}ax^4 + bx^2y^2 + cy^4 \mbox{ with }|c| \geq |a|,\\
ax^4 + bx^3y + (c+2a)x^2y^2 + bxy^3 + ay^4\mbox{ with }|4a+2b+c| \geq |4a-2b+c|.\end{cases}\]
In the case that $-\Delta(F)$ is also not a square, these inequalities must be strict by (\ref{Ci}).
\end{proof}

Lemma \ref{representative cor} implies that the condition $|v|>|a|$ in the definition of $\R_a(X)$ as in (\ref{Ra definition}) already takes $\GL_2(\bZ)$-equivalence into account. We now prove:

\begin{proposition}\label{max count}For $i=1,2$, let $N_{i,\mmax}^{(r_2)}(X)$ be the number of $\GL_2(\bZ)$-equivalence classes of forms $F\in V_{\J^{(i)}}^{\mmax}(\bZ)(X)$ such that $F(x,1)$ has exactly $4-2r_2$ real roots. Then
\begin{align*}N_{1,\mmax}^{(r_2)}(X) & = \frac{\fr(r_2)}{\zeta(2)}\frac{\sqrt{2}\Gamma(1/4)^2}{120\sqrt{\pi}}\prod_p\left(1-\frac{2p-1}{p^3}\right)X^{3/4}\log X + O(X^{3/4}(\log X)^{3/4}),\\
N_{2,\mmax}^{(r_2)}(X) & =\frac{\fr(r_2)}{\zeta(2)}\frac{\sqrt{2}\Gamma(1/4)^2}{80\sqrt{\pi}}\prod_p\left(1-\frac{2p-1}{p^3}\right)X^{3/4}\log X + O(X^{3/4}(\log X)^{3/4}),\end{align*}
where $\fr(r_2)=1$, $\fr(1) \sqrt{2}$, and $\fr(2) = 1+\sqrt{2}$ are as in Theorem \ref{real split MT}.
\end{proposition}

\begin{proof}The proof is very similar to that of Proposition \ref{N1X count}. We shall prove in Proposition \ref{error theorem} (a) that the number of forms $F\in V_{\J^{(i)}}^{\mmax}(\bZ)(X)$ with $|\Delta(F)|$ a square is an error term $O(X^{2/3})$. By Lemmas \ref{tail cut} and \ref{representative cor}, in the notation of the previous subsection, we then have
\begin{align*}
     N_{1,\mmax}^{(r_2)}(X) & = \sum_{1\leq |a|\ll X^{1/4}(\log X)^{1/4}} M_{a}^{(r_2)}(X) + O\left(X^{3/4}(\log X)^{3/4}\right),\\
     N_{2,\mmax}^{(r_2)}(X) & = \sum_{1\leq |a|\ll X^{1/4}(\log X)^{1/4}} M_{a}^{\star,(r_2)}(256X) + O\left(X^{3/4}(\log X)^{3/4}\right),
\end{align*}
where the bound on $|a|$ comes from (\ref{bound on a}). Also, by Theorems \ref{J1 thm} and \ref{J2 thm}, both $M_{a}^{(r_2)}(X)$ and $M_a^{\star,(r_2)}(X)$ are zero when $a$ is not square-free. Now, recall (\ref{precount}) and observe that
\[ \sum_{1\leq |a|\ll X^{1/4}(\log X)^{1/4}} O_\ep (X^{3/8 + \ep} + X^{3/4}|a|^{-1}) = O\left(X^{3/4}(\log X)^{3/4}\right).\]
It thus follows that
\begin{align*}
    N_{1,\mmax}^{(r_2)}(X) & = \frac{4\P}{5}\sum_{\substack{{1\leq |a|\ll X^{1/4}(\log X)^{1/4}}\\a\text{ \tiny is square-free}}}\#(\R_a^{(r_2)}(X)\cap\bZ^2) + O\left(X^{3/4}(\log X)^{3/4}\right),\\
     N_{2,\mmax}^{(r_2)}(X) & = \frac{1}{64}\cdot\frac{6\P}{5}\sum_{\substack{{1\leq |a|\ll X^{1/4}(\log X)^{1/4}}\\a\text{ \tiny is square-free}}}\#(\R_a^{(r_2)}(256X)\cap\bZ^2) + O\left(X^{3/4}(\log X)^{3/4}\right),
\end{align*}
where $\P$ is the infinite product in the statement of the proposition. As we saw in Subsection \ref{pre-count section}, the number $\#(\R_a^{(r_2)}(X)\cap\bZ^2)$ may be estimated by the area of $\R_a^{(r_2)}(X)$, which was computed in Subsection \ref{area comp}. The same calculation as in the proof of Proposition \ref{N1X count} then shows that
\begin{align*} N_{1,\mmax}^{(r_2)}(X) & = \frac{1}{\zeta(2)}\cdot\frac{4\P}{5}\cdot \frac{\fr(r_2)\sqrt{2}\Gamma(1/4)^2}{96\sqrt{\pi}} X^{3/4}\log X+
O\left(X^{3/4}(\log X)^{3/4}\right),\\
N_{2,\mmax}^{(r_2)}(X) & = \frac{1}{\zeta(2)}\cdot\frac{1}{64}\cdot\frac{6\P}{5}\cdot\frac{\fr(r_2)\sqrt{2}\Gamma(1/4)^2}{96\sqrt{\pi}} (256X)^{3/4}\log(256 X) + O\left(X^{3/4}(\log X)^{3/4}\right),\end{align*}
where the factor $1/\zeta(2)$ comes from the fact that  we are only summing over the square-free $a$. They simplify to the desired expressions.
\end{proof}

\begin{proof}[Proof of Theorem \ref{real split MT}] Recall from Proposition \ref{N3X} that the forms in $V_{\J^{(3)}}(\bZ)(X)$ contribute to an error term $O(X^{3/4})$. As we shall prove in the next subsection, the forms in $V_{\J^{(1)}}(\bZ)(X)$ and $V_{\J^{(2)}}(\bZ)(X)$ which have abelian Galois group or are reducible also have negligible contribution; see Proposition \ref{error theorem}. Hence, we deduce that
\[ N_{D_4}^{(r_2)}(X) = N_{1,\mmax}^{(r_2)}(X) + N_{2,\mmax}^{(r_2)}(X) + O\left(X^{3/4}\right),\]
and the theorem follows from adding the two formulae in Proposition \ref{max count}.
\end{proof}

\subsection{Maximal forms which are not dihedral} 
\label{max abel forms}

In this subsection, we shall complete the proof of Theorem \ref{real split MT} by showing that for $i=1,2$, the forms $F\in V_{\J^{(i)}}^{\mmax}(\bZ)$ for which 
\begin{itemize}
    \item $\Delta(F)$ or $-\Delta(F)$ is a square,
    \item $F$ is irreducible with $\Gal(F)\simeq V_4$,
    \item $F$ is irreducible with $\Gal(F)\simeq C_4$,
    \item $F$ is reducible,
\end{itemize}
contribute only to an error term. The first two cases clearly collapse into one. By (\ref{Ci}) as well as Theorems \ref{J1 thm} and \ref{J2 thm}, we have the following:

\begin{lemma}\label{square prop}For $i= 1,2$, let $F\in V_{\J^{(i)}}^{\mmax}(\bZ)$ be as in (\ref{F in J}) such that $|\Delta(F)|$ is a square. Then, we have $C = \pm A$ for $i=1$, and $C=-4A$ or $B=0$ for $i=2$.
\end{lemma}

To deal with the irreducible forms $F$ with $\Gal(F)\simeq C_4$, we use the next two lemmas:

\begin{lemma}\label{C4 vs D4}For $i=1,2,3$, an irreducible form $F\in V_{\J^{(i)}}(\bZ)$ has cyclic Galois group if and only if $\Delta(F)$ is not a square but $\C_{\J^{(i)}}(F)$ is a square in $\bQ$.
\end{lemma}
\begin{proof}The proof is analogous to that of \cite[Proposition 3.6]{TX}. Suppose that $F$ is as in (\ref{generic F}), and define the \emph{cubic resolvent polynomial} of $F$ to be
\[ R_F(x) = a_4^3X^3 - a_4^2a_2X^2 + a_4(a_3a_1-4a_4a_0)X - (a_3^2a_0 + a_4a_1^2 - 4a_4a_2a_0).\]
Since $F\in V_{\J^{(i)}}(\bZ)$, by Proposition \ref{criterion} we know that $\Gal(F)$ is small, and so $R_F(x)$ is reducible over $\bQ$. We may assume that $\Delta(F)$ is not a square, for otherwise $\Gal(F)\simeq V_4$. Then $R_F(x)$ has a unique root $r_F$ in $\bQ$, and let us put
\[ \theta_1(F) = (a_3^2 - 4a_4(a_2-r_Fa_4))\Delta(F) \AND \theta_2(F) = a_4(r_F^2a_4-4a_0)\Delta(F).\]
We have the well-known criterion, by \cite[Corollary 4.3]{Con} for example, that
\[ \Gal(F)\simeq C_4 \mbox{ if and only if }\theta_1(F),\theta_2(F)=\square\mbox{ in }\bQ.\]
Now, for $F$ having the shape in (\ref{F in J}), a simple calculation yields
\[ r_F = \begin{cases} B/A & \mbox{for $i=1$}, \\ (-1)^i2 & \mbox{for $i=2,3$},\end{cases}\AND \{\theta_1(F),\theta_2(F)\} = \{0,(B^2 - 4AC)\Delta(F)\}\]
in all three cases. The claim now follows from (\ref{Ci}).
\end{proof}


\begin{lemma}\label{C4 prop}For $i=1,2$, let $F\in V_{\J^{(i)}}^{\mmax}(\bZ)$ be as in (\ref{F in J}) such that $|\C_{\J^{(i)}}(F)|$ is a square. Let $D = \gcd(N_1(F),N_2(F))$, where 
\[ (N_1(F),N_2(F))  = \begin{cases} (A,C) & \mbox{for $i=1$},\\ (4A+C,2B) &\mbox{for $i=2$}.\end{cases} \]
Write $N_1(F) = Dn_1$ and $N_2(F)=Dn_2$. Then, there exist $s,t\in\{0,1\}$ such that
\[B^2 - 4AC  = \begin{cases}\pm 4^s n_1n_2 &\mbox{for $i=1$}, \\ \pm4^{s-t}(n_1-n_2)(n_1+n_2)& \mbox{for $i=2$}.\end{cases}\]
\end{lemma}

\begin{proof}
For $i=1$, observe that $n_1$ and $n_2$ are square-free by Theorem~\ref{J1 thm}. For $i=2$, similarly we have that $n_1-n_2$ and $n_1+n_2$ are square-free by Theorem~\ref{J2 thm}, and also note that their greatest common divisor is equal to $2^t$ for $t\in\{0,1\}$. By (\ref{Ci}), we have
\begin{equation}\label{Ci N12}\C_{\J^{(i)}}(F) = \begin{cases}
16N_1(F)N_2(F)(B^2-4AC)&\mbox{for $i=1$},\\
(N_1(F) - N_2(F))(N_1(F) + N_2(F))(B^2-4AC)&\mbox{for $i=2$}.\end{cases}\end{equation}
Since $|\C_{\J^{(i)}}(F)|$ is a square, we deduce that there exists $m\in\bZ$ such that $|m|$ is a square and
\[ B^2 - 4AC = \begin{cases} mn_1n_2& \mbox{for $i=1$},\\ 
m(n_1 - n_2)(n_1+n_2)/4^t&\mbox{for $i=2$}.\end{cases}\]
In both cases, we then have $|m|=4^s$ for $s\in\{0,1\}$ by Corollary~\ref{fund disc}, whence the claim.
\end{proof}

Finally we consider the reducible forms $F$. Remarkably $Q_F$ being maximal severely restricts the possibilities of $F$ in this case. Let us recall some terminology from \cite{TX}. By \cite[Proposition 3.11]{TX}, a reducible $F\in V_{\J^{(i)}}(\bZ)$ has two possible types, defined as follows. Recall the notation in (\ref{MJ def}) and (\ref{quadratic action}). Then, either
\[ F = m\cdot \varphi \varphi_{M_{\J^{(i)}}}  \mbox{ for some $m\in\bQ^\times$ and an integral binary quadratic form $\varphi$}, \]
in which case we say that $F$ is of \emph{type 1}, or
\[ F =\varphi\psi \mbox{ for integral binary quadratic forms $\varphi,\psi$ with $\varphi_{M_{\J^{(i)}}}=-\varphi$ and $\psi_{M_{\J^{(i)}}} = -\psi$},
\]
in which case we say that $F$ is of \emph{type 2}. We then have the following:

\begin{lemma}\label{reducible prop}For $i=1,2$, let $F\in V_{\J^{(i)}}(\bZ)$ be as in (\ref{F in J}) and reducible of non-zero discriminant.
\begin{enumerate}[(a)]
\item If $F$ is of type 1, then there exist $a,b,c,m\in\bZ$ such that $m$ divides $A,B,C$ and
\[ \begin{cases}(A,C, B^2-4AC) = (ma^2,mc^2, m^2b^2(b^2-4ac))&\mbox{for $i=1$}, \\
(4A\pm 2B+C, B^2-4AC) = (m(a\pm b+c)^2,m^2(a-c)^2(b^2-4ac))&\mbox{for $i=2$}. \end{cases}\]
\item If $F$ is of type 2, then $B^2-4AC$ is a square.
\end{enumerate}
In particular, in the case that $F\in V_{\J^{(i)}}^{\mmax}(\bZ)$, necessarily $F$ is of type $2$ and $B^2-4AC=1$.
\end{lemma}

\begin{proof}If $F$ is of type 1, then we easily see that
\[F(x,y) =\begin{cases} m(ax^2 + bxy + cy^2)(ax^2 - bxy + cy^2) & \mbox{for $i=1$},\\ 
m(ax^2 + bxy + cy^2)(cx^2 + bxy + ay^2) & \mbox{for $i=2$},
\end{cases}\]
where $m\in\bQ^\times$ and $a,b,c\in\bZ$. Changing $a,b,c\in\bZ$ if necessary, we may take $m\in\bZ$, and the claim then follows from a direct calculation. If $F$ is of type 2, then we have
\[F(x,y) =\begin{cases} 
(ax^2 + by^2)(cx^2 + dy^2) & \mbox{for $i=1$},\\
(ax^2 + bxy + ay^2)(cx^2 + dxy + cy^2) & \mbox{for $i=2$},
\end{cases}\]
where $a,b,c,d\in\bZ$, and we compute that $B^2-4AC = (ad-bc)^2$, whence the claim.\\

Suppose now that $F$ is maximal and for contradiction that it is of type 1. In the notation of part (a) above, from Theorems~\ref{J1 thm} and~\ref{J2 thm}, as well as Corollary~\ref{fund disc}, we see that
\[m=\pm1\AND\begin{cases}
 a^2 = c^2 = b^2 = 1 &\mbox{for $i=1$},\\
(a-b+c)^2 = (a+b+c)^2 = (a-c)^2 = 1 &\mbox{for $i=2$}.
\end{cases}\]
For $i=1,2$, respectively, by solving for $a,b,c\in\bZ$, we further deduce that
\[ (A,C,B) \in \{ \pm (1,1,1), \pm (1,1,-1), \pm (1,1,3), \pm (1,1,-3)\}\AND B^2-4AC=0.\]
But then either $F$ is irreducible or $F$ violates the criteria in Theorems~\ref{J1 thm} and~\ref{J2 thm}. It follows that $F$ is of type 2, and $B^2-4AC=1$ by (b) and Corollary~\ref{fund disc}.
\end{proof}

We shall now apply the above lemmas to show that the forms $F\in V_{\J^{(i)}}^{\mmax}(\bZ)(X)$ for which $|\Delta(F)|$ is a square or $\Gal(F)$ is non-dihedral are negligible.

\begin{proposition}\label{error theorem}For $i=1,2$, we have 
\begin{enumerate}[(a)]
\item $\#\{F\in V_{\J^{(i)}}^{\mmax}(\bZ)(X): |\Delta(F)|\mbox{ is a square}\} = O\left(X^{2/3}\right)$,
\item $\#\{F\in V_{\J^{(i)}}^{\mmax}(\bZ)(X): |\C_{\J^{(i)}}(F)|\mbox{ is a square}\} = O\left(X^{1/2}(\log X)^2\right)$,
\item $\#\{F\in V_{\J^{(i)}}^{\mmax}(\bZ)(X): F\mbox{ is reducible}\} = O\left(X^{1/2}(\log X)^2\right)$.
\end{enumerate}
\end{proposition}

\begin{proof}[Proof of (a)] By Lemma~\ref{square prop} and (\ref{Ci}), we have
\begin{align*}\#\{F\in V_{\J^{(1)}}^{\mmax}(\bZ)(X) : |\Delta(F)| = \square \}&\leq \#\{(A,B)\in\bZ^2: 0<|16A^2(B^2 - 4A^2)| < X\}\\
&\hspace{1cm}+\#\{(A,B)\in\bZ^2: 0<|16A^2(B^2 + 4A^2)| < X\},\\
\#\{F\in V_{\J^{(2)}}^{\mmax}(\bZ)(X) : |\Delta(F)| = \square \}&\leq\#\{(A,B)\in\bZ^2: 0<|4B^2(B^2 + 16A^2)| < X\}\\
&\hspace{1cm}+\#\{(A,C)\in\bZ^2: 0<|4AC(4A+C)^2| < X\}.\end{align*}
Now, for any integral binary cubic form $\xi(x,y)$ of non-zero discriminant,  we have
\begin{equation}\label{Mahler} \#\{(x,y)\in\bZ^2 : 0<|\xi(x,y)|<X\} = A_\xi X^{2/3} + O_\xi\left(X^{1/2}\right)\end{equation}
by a result of Mahler \cite{Mah}, where $A_\xi$ denotes the area of the region in $\bR^2$ given by $|\xi(x,y)|\leq 1$. The claim is now clear. \\
\end{proof}

\begin{proof}[Proof of (b)]From Lemma~\ref{C4 prop} and (\ref{Ci N12}), we easily see that
\begin{align*}
\#\{F\in V_{\J^{(1)}}^{\mmax}(\bZ)(X) : |\C_{\J^{(1)}}(F)|=\square\} & \ll \#\{(D,n_1,n_2)\in\bZ^3: 0<(Dn_1n_2)^2<X\},\\
\#\{F\in V_{\J^{(2)}}^{\mmax}(\bZ)(X) : |\C_{\J^{(2)}}(F)|=\square\} & \ll \#\{(D,n_1,n_2)\in\bZ^3: \\ & \hspace{1.55cm}0< (D(n_1-n_2)(n_1+n_2))^2<4X\}.
\end{align*}
They amount to counting the sum of the number of ways to write $N$ as a product of three integers, for $N$ ranging up to $X^{1/2}$. We then see that the claim holds.
\end{proof}

\begin{proof}[Proof of (c)]From Lemma~\ref{reducible prop} and (\ref{Ci}), we deduce that
\begin{align*}
\#\{F\in V_{\J^{(1)}}^{\mmax}(\bZ)(X):F\mbox{ is reducible}\}& \leq2\cdot\# \{(A,C)\in\bZ^2: 0< |16AC| < X\\ &\hspace{4.55cm}\AND 1+4AC=\square\},\\
\#\{F\in V_{\J^{(2)}}^{\mmax}(\bZ)(X):F\mbox{ is reducible}\}&\leq2\cdot\# \{(A,C)\in\bZ^2: 0< |(4A-C)^2-4| < X\}.
\end{align*}
Notice that $(1+4AC)^{1/2}\ll X^{1/2}$ for $(A,C)\in\bZ^2$ in the first set, while $4A-C\ll X^{1/2}$ for $(A,C)\in \bZ^2$ in the second set. Thus, the two numbers above on the right are bounded by
\[ \sum_{\substack{0<N\ll X^{1/2} \\ N\text{ is odd}}}  O\left(d\left(\frac{N^2-1}{4}\right)\right)\AND \sum_{0<N\ll X^{1/2}} O\left(d(N^2-4)\right),\]
respectively, where $d(\cdot)$ is the divisor function. The claim now follows from \cite[Theorem 2]{Lap}.
\end{proof}

\section{Counting binary quartic forms and elliptic curves by discriminant}
\label{disc section}

In this section, we shall prove Theorems \ref{V4 thm}, \ref{D4 thm}, and \ref{EC thm}, which involve counting integral binary quartic forms or elliptic curves by their discriminant.

\subsection{Proof of Theorem \ref{V4 thm}}

By Corollary \ref{V4 forms}, a form $ F\in V_{\J^{(1)}}(\bZ)\cup V_{\J^{(2)}}(\bZ)\cup V_{\J^{(3)}}(\bZ)$ with $\Delta(F)$ a square and $Q_F$ maximal is necessarily of the shape $ax^4 + bx^2y^2 + ay^4$. Note that such a form has discriminant $16a^2(b^2-4a^2)^2$. We make two observations:

\begin{lemma}\label{V4 lemma}Let $F(x,y) = ax^4 + bx^2y^2 + ay^4$ be an integral form of non-zero discriminant such that $Q_F$ is maximal. Then:
\begin{enumerate}[(a)]
\item $F$ is reducible if and only if $(a,b)=(0,\pm1)$.
\item $F$ is $\GL_2(\bZ)$-inequivalent to every other form in $V_{\J^{(1)}}(\bZ)\cup V_{\J^{(2)}}(\bZ)\cup V_{\J^{(3)}}(\bZ)$.
\end{enumerate}
\end{lemma}
\begin{proof}By Lemma \ref{reducible prop}, the form $F$ is reducible only if $b^2 - 4a^2 = 1$, so (a) is clear. Suppose next that $T\in \GL_2(\bZ)$ and $F_T\in V_{\J^{(i_0)}}(\bZ)$ for some $i_0=1,2,3$. This means that $TM_{\J^{(i_0)}}T^{-1}$ also fixes $F$. But $F\in V_{\J^{(1)}}(\bZ)\cup V_{\J^{(2)}}(\bZ) \cup V_{\J^{(3)}}(\bZ)$ by (\ref{F in J}), so from \cite[Theorem 1.1]{TX} we see that
\[ TM_{\J^{(i_0)}}T^{-1} = \pm M_{\J^{(i)}} \mbox{ for some }i=1,2,3,\mbox{ and in fact }i=i_0\]
because $\J^{(1)},\J^{(2)},\J^{(3)}$ have distinct discriminants. A direct computation then shows that $\pm T$ must be one of the matrices in (\ref{T choices}). It follows that $F_T = F$ and so (b) holds.
\end{proof}

Lemma \ref{V4 lemma} implies that
\begin{align*} N_{V_4}'(X) &= \#\{(a,b)\in \bZ^2\setminus\{(0,\pm 1)\}: ax^4 + bx^2y^2 + ay^4\mbox{ is maximal},\\
&\hspace{5.5cm}\mbox{and }0<|a(b-2a)(b+2a)|<X^{1/2}/4\}.\end{align*}
Now, for any odd prime $p$, it is easy to see from Theorem \ref{J1 thm} that
\[ ax^4 + bx^2y^2 + ay^4\mbox{ is maximal at $p$ if and only if }a(b-2a)(b+2a)\not\equiv0\ppmod{p^2}.\]
We may then estimate $N_{V_4}'(X)$ using the following result of Stewart and the second author in \cite[Theorem 1.2]{SX} which is analogous to (\ref{Mahler}):

\begin{proposition}\label{SX prop}Let $\xi(x,y)$ be an integral binary cubic form of non-zero discriminant, and put
\[ N_\xi'(X) =\#\{(x,y)\in\bZ^2: \xi(x,y)\mbox{ is square-free and }0<|\xi(x,y)|<X\}.\]
Suppose also that $\xi(x,y)$ is completely reducible over $\bQ$. Then, we have
\[ N_\xi'(X) = A_\xi \prod_p\left(1-\frac{\rho_\xi(p)}{p^4}\right) X^{2/3} + O_\xi\left(X^{2/3}(\log X\log\log X)^{-1}\right),\]
where $A_\xi$ denotes the area of the region in $\bR^2$ defined by $|\xi(x,y)|\leq 1$, and
\[\rho_\xi(p)  = \#\{(x,y)\in (\bZ/p^2\bZ)^2: \xi(x,y)\equiv0\ppmod{p^2}\}\]
for each prime $p$.
\end{proposition}

It is well-known that
\[A_\xi = \frac{1}{\Delta(\xi)^{1/6}}\frac{3\Gamma(1/3)^2}{\Gamma(2/3)}\mbox{ when }\Delta(\xi)>0.\]
Further, the proof of Proposition~\ref{SX prop} given in \cite{SX} is still valid when the condition that $\xi(x,y)$ is square-free at $p$, namely $\xi(x,y)\not\equiv0$ (mod $p^2$), is replaced by another condition mod $p^2$ at finitely many primes $p$. Let us apply Proposition \ref{SX prop} to the form
\[\xi_0(a,b) = a(b-2a)(b+2a),\mbox{ where }\Delta(\xi_0)=2^4,\]
but with the condition at $2$ replaced by that $ax^4 + bx^2y^2 + ay^4$ is maximal at $2$. We then obtain
\[ N_{V_4}'(X) = \frac{3\Gamma(1/3)^2}{2^{2/3}\Gamma(2/3)}\prod_p\left(1-\frac{\rho_{\xi_0}(p)}{p^4}\right)\left(\frac{X^{1/2}}{4}\right)^{2/3} + O\left(X^{1/3}(\log X\log\log X)^{-1}\right),\]
where the values of $\rho_{\xi_0}(p) = \rho_0(p)$ are given in Proposition~\ref{rho eva}. The theorem now follows.

\subsection{Proof of Theorem~\ref{D4 thm}} 


We first provide the upper bound. Note that it suffices to prove, for each $i=1,2,3$, an upper bound for the set
\[V_{\J^{(i)}}(\bZ)'(X)  = \{F \in V_{\J^{(i)}}(\bZ) : 0 < |\Delta(F)| < X \}. \]
Analogous to the conductor case, by (\ref{Ci xy formula}) this problem reduces to counting integral points in a $2$-dimensional region, but this time we consider
\[ S(X) = \{(x,y)\in \bR^2 : |y|,|x^2-y|\geq 1,\, |y(x^2-y)^2| < X\}.\]
Each point $(x,y)$ is weighted by the divisor function $d(y)$ for $i=1,2$, and the function $d_\square(y)$ in (\ref{kappa}) for $i=3$. In particular, we have
\begin{align*} \#V_{\J^{(1)}}(\bZ)'(X),\#V_{\J^{(2)}}(\bZ)'(X) &\ll \sum_{(x,y)\in S(X)\cap \bZ^2} d(y)\\ \#V_{\J^{(3)}}(\bZ)'(X)&\ll\sum_{(x,y)\in S(X)\cap \bZ^2}d_\square(y)\end{align*}
similar to Lemma \ref{RX bijections}.

\begin{remark}\label{disc remark}
Let us briefly discuss why the analogue of Lemma \ref{tail cut} cannot be expected to hold for $S(X)$. In the proof of Lemma \ref{tail cut} the crucial observation is that the length of the intervals given by Lemma \ref{bound on y} are $O(X/x^2)$, and by Lemma \ref{x bd}, one sees that this is bounded away from zero. The analogous intervals in the discriminant case only have length $O(X/x^4)$, and thus for each $x \gg X^{1/4}$ there are at most one $y$ such that $(x,y) \in S(X)$. Therefore we cannot hope to average the divisor function as in the proof of Lemma \ref{tail cut} using (\ref{dr growth}). \end{remark}

As a result we are only able to use the absolute upper bound for the divisor function, which is
\[ d(y) = y^{1/\log\log y}.\]
This bound holds for $d_\square(y)$ as well, so we may apply this deal with all $i = 1,2,3$, and obtain
\begin{equation}\label{D4 disc bound} \#V_{\J^{(i)}}(\bZ)'(X) \ll \#(S(X)\cap \bZ^2)\cdot X^{1/\log\log X}. \end{equation}
The number of integral points in $S(X)$ may be estimated using Davenport's lemma, as follows. 

\begin{proposition} \label{premax disc} We have
\[ \#(S(X)\cap \bZ^2) = \Area(S_0(X)) + O \left(X^{1/2} \right),\]
where $S_0(X)$ is the region 
\[ S_0(X) = \{(x,y) \in S(X) : |x| \leq X^{1/4} \}.\]
\end{proposition} 

\begin{proof} For any $(x,y)\in S(X)\cap \bZ^2$, the same proof in Lemma \ref{x bd} shows that $|x| \ll X^{1/2}$. For $|x|\in (X^{1/4},X^{1/2}]$, as noted in Remark \ref{disc remark} there are at most $O(1)$ possibilities for $y$, so the total contribution from this range is $O \left(X^{1/2}\right)$. Thus, it remains to consider integral points in $S_0(X)$. The claim now follows from Proposition \ref{Davenport} because for any $(x,y)\in S_0(X)$, we have
\[ |x| = O\left(X^{1/4}\right) \AND |y| = O\left(X^{1/2}\right)\]
for the $1$-dimensional projections. To see why the latter hound holds, suppose that $|y|\gg X^{1/2}$. Then for the inequality $|y(x^2 - y)^2| < X$ to hold we must have $|x^2 - y| \ll X^{1/4}$, which implies that $y = x^2 + O (X^{1/4})$. But $x^2 \ll X^{1/2}$, whence $|y| \ll X^{1/2}$ and this is a contradiction.
\end{proof}

To compute the area of $S_0(X)$, we shall require the following lemma, which is analogous to Lemma \ref{bound on y}. As one can easily verify by a direct computation, we have:

\begin{lemma}\label{bound on y'}Let $(x,y)\in \bR^2$. 
\begin{enumerate}[(a)]
\item For $|y|\geq X^{1/3}$, the condition $|y(x^2-y)^2|<X$ is equivalent to
\[ y>0 \AND \sqrt{y-\frac{X^{1/2}}{y^{1/2}}} < |x| < \sqrt{y+\frac{X^{1/2}}{y^{1/2}}}.\]
\item For $|y| < X^{1/3}$ and $y\neq0$, the condition $|y(x^2-y)^2|<X$ is equivalent to
\[ 0\leq |x| < \sqrt{y+\frac{X^{1/2}}{|y|^{1/2}}}.\]
\end{enumerate}
\end{lemma}

\begin{lemma} \label{S0 area} We have
\[\Area(S_0(X)) =O \left(X^{1/2} \log X\right). \]
\end{lemma}

\begin{proof}As shown in the proof of Proposition \ref{premax disc}, we have $|y| \ll X^{1/2}$ for all $(x,y)\in S_0(X)$. Lemma \ref{bound on y'} then implies
\[ \Area(S_0(X)) \ll \int_{1\leq |y|< X^{1/3}}\sqrt{y+\frac{X^{1/2}}{|y|^{1/2}}}dy + \int_{X^{1/3}}^{X^{1/2}}\left(\sqrt{y+\frac{X^{1/2}}{y^{1/2}}}-\sqrt{y-\frac{X^{1/2}}{y^{1/2}}}\right)dy.\]
For $|y|<X^{1/3}$, note that
\[\sqrt{y+\frac{X^{1/2}}{|y|^{1/2}}} - \sqrt{\frac{X^{1/2}}{|y|^{1/2}}} = \frac{y}{\sqrt{y+X^{1/2}/|y|^{1/2}} + \sqrt{X^{1/2}/|y|^{1/2}}} = O \left(\frac{y^{5/4}}{X^{1/4}} \right). \]
Hence, we have
\[\int_{1\leq |y|< X^{1/3}}\sqrt{y+\frac{X^{1/2}}{|y|^{1/2}}}dy = \int_{1\leq |y|<X^{1/3}}\left(\frac{X^{1/4}}{|y|^{1/4}}+O\left(\frac{y^{5/4}}{X^{1/4}}\right)\right) = O\left(X^{1/2}\right).\]
For $y\geq X^{1/3}$, similarly note that
\[\sqrt{y + \frac{X^{1/2}}{y^{1/2}}} - \sqrt{y - \frac{X^{1/2}}{y^{1/2}}} = \frac{2X^{1/2}/y^{1/2}}{\sqrt{y + X^{1/2}/y^{1/2}} + \sqrt{y - X^{1/2}/y^{1/2}}} = O \left(\frac{X^{1/2}}{y}\right).\]
This implies that
\[\int_{X^{1/3}}^{X^{1/2}} \left(\sqrt{y + \frac{X^{1/2}}{y^{1/2}}} - \sqrt{y - \frac{X^{1/2}}{y^{1/2}}}\right)dy = O \left(X^{1/2} \log X \right).\]
The claim now follows.
\end{proof}

Let us remark that we shall compute the area of a similar region in Subsection \ref{EC pf} in order to prove Theorem \ref{EC thm}, with the difference being that the coefficient of the main term $X^{1/2} \log X$ can be computed exactly. \\

From (\ref{D4 disc bound}), Proposition \ref{premax disc}, and Lemma \ref{S0 area}, we now deduce that
\[ N_{D_4}'(X) \ll \left(\Area(S_0(X)) + O\left(X^{1/2}\right)\right)\cdot X^{1/\log\log X} \ll X^{1/2+1/\log\log X}\log X,\]
which is the desired upper bound.\\

We now move on to prove the lower bound. Note that it suffices to consider one of the three families, and we shall restrict our attention to the forms in
\[V_{\J^{(1)}}^{\mmax}(\bZ)'(X) = \{F \in V_{\J^{(1)}}^{\mmax}(\bZ): 0<|\Delta(F)|<X\}.\]
Notice that for the irreducible $F$, those with $\Gal(F)\simeq V_4$ only have $O(X^{1/3})$ contribution by Theorem \ref{V4 thm}, while those with $\Gal(F)\simeq C_4$ have $O(X^{1/2}\log X)$ contribution by Lemma \ref{C4 vs D4} and the next estimate analogous to Proposition \ref{error theorem} (b).

\begin{proposition}\label{C4 disc} We have
\[ \#\{F\in V_{\J^{(1)}}^{\mmax}(\bZ)'(X): |\C_{\J^{(1)}}(F)|=\square \} = O\left(X^{1/2}\log X\right).\]
\end{proposition}
\begin{proof}As in the proof of Proposition \ref{error theorem} (b), from Lemma \ref{C4 prop} and (\ref{Ci}), we see that
\begin{align*}
\#\{F\in V_{\J^{(1)}}^{\mmax}(\bZ)'(X) : |\C_{\J^{(1)}}(F)|=\square\} & \ll \#\{(D,n_1,n_2)\in\bZ^3: 0<D^2|n_1n_2|^3<X\}.
\end{align*}
The right hand side is in turn bounded by
\[\sum_{1 \leq D \leq X^{1/2}} \sum_{1 \leq N < 3X^{1/3}/D^{2/3}} O\left(d(N) \right) = O \left(X^{1/2} \log X\right),\]
where $d(\cdot)$ is the divisor function.
\end{proof}

Instead of the entire $V_{\J^{(1)}}^{\mmax}(\bZ)$, we shall in fact only look at those forms $F$ whose coefficients satisfy specific conditions. Note that Theorem \ref{J1 thm} and Lemma \ref{reducible prop} immediately imply:

\begin{lemma}\label{lower bound lem}
Let $F\in V_{\J^{(1)}}(\bZ)$ be as in (\ref{F in J}). 
\begin{enumerate}[(a)]
    \item If both $AC$ and $B^2-4AC$ are square-free, then $F$ is maximal at all odd primes $p$.
    \item If $F$ is maximal and $B^2-4AC\neq 1$, then $F$ is irreducible.
\end{enumerate}
\end{lemma}

Insisting that $F$ is also maximal at the prime $2$ only changes the asymptotic count by some constant factor. Moreover, the $\GL_2(\bZ)$-equivalence class of any $F\in V_{\J^{(1)}}(\bZ)$ contains at most two distinct forms in $V_{\J^{(1)}}(\bZ)$ by Lemma~\ref{representative cor}. Recalling (\ref{Ci}), from Proposition \ref{C4 disc} and Lemma \ref{lower bound lem}, we then deduce that 
\begin{align*} N_{D_4}'(4X) + O\left(X^{1/2}\log X\right)&\gg \#\{(a,b,c)\in\bZ^3: b^2-4ac\neq1,\mbox{ both }ac\mbox{ and }b^2-4ac\\&\hspace{3cm}\mbox{are square-free, }0<|4ac(b^2-4ac)^2|<X\},\end{align*}
where $4X$ is used only for convenience.  Since we are looking for a lower bound, it is permissible to assume further that
\[ 4ac \geq X^{1/3} \AND ac \ll_\delta X^{1/3 + \delta}\mbox{ for some }\delta >0.\] 
By Lemma~\ref{bound on y'}, the condition $|4ac(b^2-4ac)^2|<X$ is equivalent to $|b|\in J_{4ac}$, where
\[J_y = \left(\sqrt{y - \frac{X^{1/2}}{y^{1/2}}}, \sqrt{y + \frac{X^{1/2}}{y^{1/2}}} \text{ } \right)\mbox{ for }y\geq X^{1/3}.\]
Setting $y=4ac$, we then deduce that
\begin{equation}\label{lower bound} N_{D_4}'(4X) + O\left(X^{1/2}\log X\right)\gg \sum_{\substack{{X^{1/3}<y<X^{1/3+\delta}}\\y\in 4\bZ,y+1\neq\square}} \mu(y/4)^2d(y/4) \sum_{x\in J_y} \mu(x^2 -y)^2,\end{equation}
where $\delta>0$ is to be chosen and $\mu(\cdot)$ denotes the M\"{o}bius function. We shall estimate this double sum using the following result due to Friedlander and Iwaniec \cite{FI}.

\begin{proposition}\label{FI estimate}For $y>0$, there exists a constant $c_y>0$ such that
\[\sum_{x<Y}\mu(x^2 - y)^2 = c_yY + O(d(y)(y+Y^2)^{1/3}\log(2yY)).\]
Moreover, there exists a constant $c_0>0$ such that $c_y\geq c_0$ for all $y>0$.
\end{proposition}
\begin{proof}See \cite[Theorem 2.1]{FI}; the claim there is stated for $x^2+y$, but essentially the same  argument yields the above asymptotic formula.
\end{proof}

Notice that $X^{1/2}/y^{1/2} < y$ when $y>X^{1/3}$. From Proposition~\ref{FI estimate}, we then deduce that the double sum on the right hand side in (\ref{lower bound}) is equal to
\begin{equation}\label{sum} \sum_{\substack{{X^{1/3}<y<X^{1/3+\delta}}\\y\in 4\bZ,y+1\neq\square}}\mu(y/4)^2d(y/4)d(y)c_y|J_y| + O\left(\sum_{X^{1/3}<y<X^{1/3+\delta}}d(y)^2y^{1/3}\log X\right),\end{equation}
where $|J_y|$ denotes the length of $J_y$. Since $|J_y| = O\left(X^{1/2}/y\right)$,
we have
\[\sum_{\substack{{X^{1/3}<y<X^{1/3+\delta}}\\y\in 4\bZ,y+1\neq\square}}\mu(y/4)^2d(y/4)c_y|J_y| \gg X^{1/2}\sum_{\substack{{X^{1/3}<y<X^{1/3+\delta}}\\y\in 4\bZ,y+1\neq\square}}\frac{\mu(y/4)^2d(y/4)}{y}\gg X^{1/2}(\log X)^2,\]
where the latter may be obtained using partial summation. Since $d(y)^2\log X = O_\ep(X^\ep)$, by choosing $\delta=1/27$ and $\ep = 1/162$ say, we see that the second sum in (\ref{sum}) is $O(X^{1/2})$ and hence is an error term. This completes the proof of the theorem.


\subsection{Proof of Theorem \ref{EC thm}} 
\label{EC pf} 

By definition, we have
\[ N_E'(X) = \{(a,b)\in\bZ^2 : p^2\mid a \mbox{ implies $p^4\nmid b$ for all primes $p$, and $0<|b^2(a^2-4b)|<X$}\}.\]
Let us first count the number of integral points in
\[ \S'(X) = \{(x,y)\in\bR^2 :|b|,|a^2-4b|\geq 1,\, |b^2(a^2-4b)|<X\}.\]
We may do so using Davenport's lemma. Indeed, for any $(a,b)\in S'(X)$, we have
\begin{equation}\label{ab bound}|b| = O\left(X^{1/2}\right)\AND |a| = O\left(X^{1/2}\right),\end{equation}
where the latter follows from a similar proof as in Lemma \ref{x bd}. Proposition \ref{Davenport} then yields
\[ \#(S'(X)\cap \bZ^2) = \Area(S'(X)) + O\left(X^{1/2}\right) = \frac{2}{3}X^{1/2}\log X,\] 
where the area of $S'(X)$ is computed in the next lemma:

\begin{lemma}\label{S' area}We have
\[ \Area(S'(X)) = \frac{2}{3}X^{1/2}\log X + O\left(X^{1/2}\right). \]
\end{lemma}
\begin{proof}An analogous statement as Lemma \ref{bound on y'} is true for the inequality $|b^2(a^2-4b)|<X$. In particular, it is not hard to see that
\begin{align}\label{S' int} \Area(S'(X)) & = \int_{1\leq |b|< (X/4)^{1/3}}\sqrt{4b + \frac{X}{b^2}}db \\\notag
&\hspace{2cm}+ \int_{(X/4)^{1/3}\leq b<X^{1/2}}\left(\sqrt{4b+\frac{X}{b^2}} - \sqrt{4b-\frac{X}{b^2}}\right)db.\end{align}
Now, by a similar calculation as in the proof of Lemma \ref{S0 area}, the first integral in (\ref{S' int}) is equal to
\[ \int_{1\leq |b| <(X/4)^{1/3}}\left(\frac{X^{1/2}}{|b|} + O\left(\frac{b^2}{X^{1/2}}\right)\right)db = \frac{2}{3}X^{1/2}\log(X) + O\left(X^{1/2}\right),\]
while the second integral in (\ref{S' int}) is bounded by
\[\int_{(X/4)^{1/3}\leq b<X^{1/2}}O\left(\frac{X}{b^{5/2}}\right)db = O\left(X^{1/2}\right).\]
Combining them together then yields the claim.
\end{proof}

Let us now apply a sieve to get the exact count for $N_E'(X)$. For $m\in\bN$, define
\[ M_E(m;X) = \#\{(a,b)\in\bZ^2: m^2\mid a,\,m^4\mid b,\, 0<|b^2(a^2-4b)|<X\}.\]
By (\ref{ab bound}) and (\ref{int cong}), we have
\begin{align*} M_E(m;X) & = 
\frac{1}{m^6}\cdot\#(S'(X)\cap\bZ^2) + O\left(\frac{X^{1/2}}{m^2}+\frac{X^{1/2}}{m^4}+1\right)\\
& = \frac{1}{m^6}\cdot\frac{2}{3}X^{1/2}\log X + O\left(\frac{X^{1/2}}{m^2}+1\right),
\end{align*}
where the second equality follows from \ref{S' area}. By the inclusion-exclusion principle, we have
\[\label{EC sieve} N_E'(X) = \sum_{\substack{m\in\bN\\p\mid m\Rightarrow p\leq (\log X)^{1/5}}} \mu(m) M_E(m;X) + O \left(\sum_{p > (\log X)^{1/5}} M_E(p;X) \right),\]
where $\mu(\cdot)$ is the M\"{o}bius function. Note that (\ref{ab bound}) implies $M_E(m;X)$ is zero for all $m > X^{1/8}$. On the one hand, we have
\[ \sum_{(\log X)^{1/5}<p<X^{1/8}}M_E(p;X) = O\left( \sum_{(\log X)^{1/5}<p<X^{1/8}}\left(\frac{X^{1/2}\log X}{p^6}+\frac{X^{1/2}}{p^2}\right)\right) = O\left(X^{1/2}\right).\]
On the other hand, we have
\[ \sum_{\substack{1\leq m <X^{1/8} \\ p \mid m \Rightarrow p < (\log X)^{1/5}}} \mu(m)M_E(m;X) = \frac{2}{3}X^{1/2}\log X\sum_{\substack{1\leq m <X^{1/8} \\ p \mid m \Rightarrow p < (\log X)^{1/5}}}\frac{\mu(m)}{m^6} + O\left(X^{1/2}\right).\]
It is standard to show that
\[ \sum_{\substack{1\leq m < X^{1/2}\\ p\mid m\Rightarrow p < (\log X)^{1/5}}}\frac{\mu(m)}{m^6} = \prod_{p}\left(1-\frac{1}{p^6}\right) + O\left( (\log X)^{-1}\right) = \frac{1}{\zeta(6)} + O\left((\log X)^{-1}\right).\]
The theorem now follows. 

\subsection*{Acknowledgments}

Part of this research was done while the first author was visiting the second author at the Mathematical Institute at the University of Oxford. She would like to thank the institute for its hospitality during her stay. The visit was supported by the China Postdoctoral Science Foundation Special Financial Grant (Award No.: 2017T100060). We thank the anonymous referees for providing insightful comments which lead to significant improvements of the paper.

\end{document}